\documentclass{amsart}
\usepackage[utf8]{inputenc}

\usepackage{amsthm,amssymb,amsmath,amsfonts}
\usepackage{mathtools}
\usepackage{tikz}
\usetikzlibrary{backgrounds,patterns,calc}
\usepackage{tikz-cd}
\tikzset{
  commutative diagrams/.cd, 
  arrow style=tikz, 
  diagrams={>=stealth}
}
\usepackage{hyperref}
\usepackage{cleveref}

\usepackage[colorinlistoftodos,prependcaption, textwidth=3cm,]{todonotes}

\DeclareMathOperator{\Ext}{Ext}
\DeclareMathOperator{\Hom}{Hom}
\DeclareMathOperator{\End}{End}
\newcommand{\K}{\mathbf{k}}

\newcommand{\cA}{\mathcal{A}}
\newcommand{\cB}{\mathcal{B}}
\newcommand{\cC}{\mathcal{C}}
\newcommand{\cX}{\mathcal{X}}
\newcommand{\qA}{\mathbb{A}}
\newcommand{\tors}{\mathcal{T}}
\newcommand{\torsfree}{\mathcal{F}}
\newcommand{\torspair}{(\tors,\,\torsfree)}

\DeclareMathOperator{\Modf}{mod}

\newcommand{\tfunc}{\mathfrak{t}}
\newcommand{\ffunc}{\mathfrak{f}}
\newcommand{\preproj}{\mathcal{P}}
\newcommand{\preinj}{\mathcal{Q}}
\newcommand{\regular}{\mathcal{R}}
\DeclareMathOperator{\add}{add}
\newcommand{\homort}{{^{\perp_H}}}
\newcommand{\extort}{{^{\perp_E}}}
\DeclareMathOperator{\gen}{Gen}
\DeclareMathOperator{\sub}{Sub}
\newcommand{\univTorsfree}{\mathcal{E}}
\newcommand{\cotorsfunc}{\mathfrak{c}}

\newcommand{\tube}{\mathbb{T}}
\newcommand{\ray}[1]{\mathbf{R}_{#1}}
\newcommand{\coray}[1]{\mathbf{C}_{#1}}
\newcommand{\wing}[2]{\mathbf{W}_{#1}^{#2}}

\newcommand{\ind}[2]{M_{#1}^{#2}}
\usepackage{MnSymbol}

\newtheorem{thm}{Theorem}[section]
\newtheorem*{thm*}{Theorem}
\newtheorem{theoremIntro}{Theorem}
\newtheorem{corol}[thm]{Corollary}
\newtheorem{lemma}[thm]{Lemma}

\theoremstyle{definition}
\newtheorem{definition}[thm]{Definition}
\newtheorem{examplex}[thm]{Example}
\newenvironment{example}
{\pushQED{\qed}\examplex}
{\popQED\endexamplex}
\newtheorem{remark}[thm]{Remark}

\newcommand{\introheaderOne}[1]{\vspace{1em}{\bfseries \noindent #1 \ }}
\newcommand{\introheaderTwo}[1]{\par\vspace{1em} \noindent\emph{#1}\ }

\title[Universal Extensions and Orthogonal Complements]{Universal Extensions and Ext-Orthogonal Complements of Torsion Classes}
\author[E. S. Rundsveen]{Endre Sørmo Rundsveen}
\address{Department of mathematical sciences, NTNU, NO-7491 Trondheim, Norway}
\email{endre.s.rundsveen@ntnu.no}

\begin{document}
\begin{abstract}
    We show that torsion pairs in Krull--Schmidt abelian categories induce an equivalence between the subcategory of torsion-free objects admitting universal extensions to the torsion subcategory, and a quotient of the ext-orthogonal complement of the torsion subcategory.

    This generalize an equivalence described by Bauer--Botnan--Oppermann--Steen for tilting-torsion pairs and by Buan--Zhou for functorially finite torsion pairs. The result also provides a more direct proof of the functorially finite case, not relying on the machinery of two-term silting complexes. 
    
    We illustrate the equivalence in the special case of tube categories.

    After first appearing on the arXiv, the author learnt that the equivalence has previously been described by Demonet--Iyama.
\end{abstract}

\maketitle

\vspace{-2em}\tableofcontents

\section{Introduction}
\introheaderOne{Historical Remarks and Motivation.}
\emph{Torsion theories} have played a prominent role in algebraic studies from the early ventures into the world of groups and modules. After the general theory was first axiomatized for abelian categories by Dickson \cite{dickson1966torsion} the theory became an integral component in the study of categories as well. The underlying idea of the theory is to partition the objects of the category into subcategories which have a nicely behaved interplay, and from that gaining a deeper understanding of the category as a whole.

The rise of \emph{tilting theory} in the years after, showcased the usefulness of torsion theory. A \emph{tilting module} $T$ of an algebra $A$ was shown to induce a special torsion pair that could be carried over to a split torsion pair in the tilted algebra $B=\operatorname{End}(T)$ \cite{BB80,Bon81,HR82}. Thus, relating structure in the algebras to each other. The theory of tilting was however somewhat limiting when venturing beyond hereditary algebras, whence Adachi, Iyama and Reiten generalized it to support $\tau$-tilting in \cite{adachi_iyama_reiten_2014}. This generalization cemented the link to torsion theories by giving a one-to-one correspondance between support $\tau$-tilting modules and functorially finite torsion classes.

In recent years, the authors of \cite{bauer2020cotorsion} showed that tilting in abelian categories with enough projectives can be described through \emph{cotorsion--torsion triples} $(\mathcal{C},\mathcal{T},\mathcal{F})$, in an effort to reduce the complexity of \emph{persistence modules}. These triples were shown to induce an equivalence between the subcategory of torsion-free objects $\mathcal{F}$ and a certain quotient of $\mathcal{C}$. Related work for module categories of artin algebras by Buan and Zhou \cite{Buan2021tauTilting} extended this equivalence by introducing \emph{left weak cotorsion--torsion triples}. These triples are in correspondance with support $\tau$-tilting modules. Around the same time \cite{asadollahi2022tau} introduced \emph{$\tau$-cotorsion--torsion triples} for abelian categories with enough projectives, which in module categories of artin algebras coincide with left weak cotorsion--torsion triples.

The equivalence of \cite{bauer2020cotorsion} and most directly \cite{Buan2021tauTilting} can be seen as a result on functorially finite torsion classes with no mention of ($\tau$-)tilting. That is, let $\torspair$ be a functorially finite torsion pair in the module category of an artin algebra $A$, then there is an equivalence (\cite[Cor.~4.8]{Buan2021tauTilting} and \cite[Thm.~2.13]{bauer2020cotorsion})
\begin{equation}\label{eq:equivalenceToGeneralize}
    \frac{\extort\tors}{\extort\tors\cap\tors}\simeq \tors\homort=\torsfree
\end{equation}
where $\extort\tors=\{M\in \Modf(A)\,|\,\Ext_A^1(M,\tors)=0\}$. The motivating question for this article is whether we can eliminate the assumption of functorially finiteness while still obtaining an equivalence.

\introheaderOne{Our Results.}
The answer to our question turns out to be closely related to the existence of universal extensions. A universal extension is a form of subcategory approximation which were for example essential in Bongartz' work on completing partial tilting modules \cite{Bon81}, and used by Kerner to study wild hereditary algebras with three non-isomorphic simples \cite{otto1990universal}. Specifically, a universal extension from an object $Y$ to a subcategory $\cX$ is a short exact sequence
\[
\begin{tikzcd}
    X\rar[tail]&E\rar[two heads]&Y
\end{tikzcd}
\]
with $X$ in $\cX$, such that every other short exact sequence starting in $\cX$ and ending in $Y$ can be realised as a pushout of it.

Our main result is the following.
\begin{theoremIntro}[\Cref{Theorem: Main result}]\label{thm: Main Theorem}
    Let $\torspair$ be a torsion pair in a Krull--Schmidt abelian category $\cA$, and $\mathcal{E}$ the subcategory of objects in $\torsfree$ that admit an universal extension to $\tors$. Then there is an equivalence
    \begin{equation}\label{eq:generalizedEquivalence}
        \frac{\extort\tors}{[\tors]}\simeq \univTorsfree.
    \end{equation}
    where $[\tors]$ is the ideal of morphisms factoring through $\tors$.
\end{theoremIntro}

\begin{remark}
    After the paper was first put on arXiv, Yu Zhou made me aware of Theorem~4.1 in \cite{DI16} which has \Cref{Theorem: Main result} as a special case. Thus, the result lacks the novelty it was first believed to possess. 
\end{remark}

When $\tors$ is functorially finite, every object of $\torsfree$ admits a universal extension. Moreover, in that case every morphism from $\extort\tors$ to $\tors$ factors through $\extort\tors\cap\tors$. We observe therefore that for functorially finite torsion pairs the equivalence in \eqref{eq:generalizedEquivalence} specializes back to that in \eqref{eq:equivalenceToGeneralize}, see  \Cref{Lem:func fin tors give old equivalence}. This provide an alternative proof of \cite[Cor.~4.8]{Buan2021tauTilting} which does not rely on the bigger machinery at play in that article.

\vspace{1em}At the end of the article we take a look at tame hereditary algebras in order to provide an example of non-functorially finite torsion classes. Utilizing a description of torsion classes of tubes given in \cite{baur2014torsion} we observe that for a torsion pair $\torspair$, all but finitely many indecomposable torsion-free objects admit a universal extension to $\tors$. In particular, when $\torspair$ is non-functorially finite we observe that within a tube $\tube$, either $\torsfree\cap\tube$ is of infinite type and every module admits a universal extension to $\tors$ (see \Cref{subsec: tors finite type}) or $\torsfree\cap\tube$ is of finite type and not all modules admit a universal sequence to $\tors$ (see \Cref{subsec: tors infinite type}). We will illustrate both cases with explicit examples, showcasing that within the tube the equivalence behaves almost as in the functorially finite case.

\introheaderOne{Conventions and notation}
Every subcategory in the article are assumed to be additive and closed under both summands and isomorphisms. In an abelian category $\mathcal{A}$ we denote by $\Ext_\mathcal{A}^i(X,Y)$ the Yoneda extension group of $i$-extensions from $X$ to $Y$, and identify them with the derived functors along the natural isomorphism whenever they are defined. Given a subcategory $\mathcal{X}$ of an abelian category $\mathcal{A}$ we denote by $\homort\mathcal{X}$ ($\extort\mathcal{X}$) the subcategory given by objects $A\in \mathcal{A}$ such that $\Hom_\mathcal{A}(A,\mathcal{X})=0$ ($\Ext_\mathcal{A}^1(A,\mathcal{X})=0$). The symmetrical notation $\mathcal{X}\homort$ and $\mathcal{X}\extort$ denote the dual cases. We say that a Krull--Schmidt category is of finite type if it has finitely many indecomposables up to isomorphism.

\subsection*{Acknowledgements} The author would like to express their sincere gratitude to Steffen Oppermann for the support and guidance provided during the project. In addition, Jacob Fjeld Grevstad is owed thanks for entertaining several discussions and for giving valuable feedback during the preparation of the manuscript.

The article is written during the authors employment as a PhD candidate at the Department of Mathematical Sciences of the Norwegian University of Science and Technology (NTNU).

Lastly, the author thanks Yu Zhou for making them aware of the paper \cite{DI16} and specifically Theorem~4.1 of that paper. 

\section{Preliminaries}

The question posed in this article is about properties of torsion subcategories in abelian categories. We therefore take the time to repeat the definition and basic properties of these. The notion of a torsion pair of an abelian category was introduced in \cite{dickson1966torsion} as a generalization of torsion in abelian groups, and can be formulated as follows in a general abelian category.

\begin{definition}
    Let $\mathcal{A}$ be an abelian category and $\torspair$ a tuple of full subcategories of $\mathcal{A}$. We call $\torspair$ a torsion pair in $\mathcal{A}$ if
    \begin{itemize}
        \item $\Hom_\Lambda(\tors,\,\torsfree)=0$, and
        \item for any object $X\in \mathcal{A}$ there exists a short exact sequence
        \[
        \begin{tikzcd}
            \tfunc X \arrow[tail]{r} &X\arrow[two heads]{r}&\ffunc X
        \end{tikzcd}
        \]
        such that $\tfunc X\in \tors$ and $\ffunc X\in \torsfree$.
    \end{itemize}
    The respective subcategories $\tors$ and $\torsfree$ of a torsion pair is called a torsion subcategory and a torsion-free subcategory. Whenever there is no ambiguity of which torsion pair we consider, an object $X$ will be called torsion if it lie in $\tors$ and torsion-free if it lie in $\torsfree$.
\end{definition}
\begin{remark}
The short exact sequence is unique for $X$ in the sense that if there is an exact sequence $\begin{tikzcd}[column sep=1.5em]
    T\arrow[tail]{r}{}&X\arrow[two heads]{r}{}&F
\end{tikzcd}$ with $T\in\tors$ and $F\in \torsfree$, then it is isomorphic to the one above. Hence we will be referring to this sequence as the \emph{torsion sequence} of $X$ with respect to $\torspair$. 
Consider a morphism $f\in\Hom_\cA(X,Y)$, then we can construct the following commutative diagram with exact rows
\[\begin{tikzcd}[column sep=2em,row sep=1.5em]
	{\tfunc X} & X & {\ffunc X} \\
	{\tfunc Y} & Y & {\ffunc Y}
	\arrow[tail, from=1-1, to=1-2]
	\arrow["{\tfunc f}", dashed, from=1-1, to=2-1]
	\arrow[two heads, from=1-2, to=1-3]
	\arrow["f", from=1-2, to=2-2]
	\arrow["{\ffunc f}", dashed, from=1-3, to=2-3]
	\arrow[tail, from=2-1, to=2-2]
	\arrow[two heads, from=2-2, to=2-3]
\end{tikzcd}\]
where the dashed morphisms are given uniquely through the universality of kernels and cokernels. Hence, the torsion pair $\torspair$ give rise to functors $\tfunc\colon \cA\to \tors$ and $\ffunc\colon \cA\to \torsfree$.
\end{remark}

In the study of subcategories of abelian categories the notion of (minimal) approximations introduced in \cite{auslander1980preprojective} has a fundamental role. A morphism $f\in \Hom_\cA(X,Y)$ is \emph{right minimal} if any endomorphism $g\in\End_\cA(X)$ such that $f\circ g=f$ is an automorphism, and equivalently it is \emph{left minimal} if any endomorphism $h\in \End_\cA(Y)$ such that $h\circ f=f$ is an autmorphism.

\begin{definition}
    Let $\cA$ be an abelian category and $\cX$ a subcategory of $\cA$. Then a morphism $f\in \Hom_\Lambda(X,A)$ with $X\in\cX$ and $A\in\cA$ is a \emph{right $\cX$-approximation} of $A$ if 
    $$
    \begin{tikzcd}
        \Hom_\cA(X',\,X)\arrow{r}{f\circ -}&\Hom_\cA(X',\,A)
    \end{tikzcd}
    $$
    is an epimorphism for all $X'\in \cX$. Equivalently, if any morphism $X'\to A$ with $X'\in \cX$ factors through $f$.
        \[
        \begin{tikzcd}
            &X'\arrow[dashed]{dl}{}\arrow{d}\\
            X\arrow{r}{f}&A
        \end{tikzcd}
        \]
        If $f$ in addition is right minimal, we call it a right minimal $\cX$-approximation of $A$. If all objects $A\in\cA$ admits a right $\cX$-approximation then $\cX$ is called \emph{contravariantly finite}. The notions of left approximations and covariantly finite subcategories are dually defined.
    A subcategory which is both covariantly and contravariantly finite is called \emph{functorially finite}.
\end{definition}

The torsion sequences for a torsion pair $\torspair$ provide approximations making $\tors$ contravariantly finite in $\cA$ and $\torsfree$ covariantly finite in $\cA$. In the module categories of a finite dimensional $\K$-algebra $\Lambda$ we have the following description of all functorially finite torsion classes:
\begin{thm}[{\cite[Theorem]{smalo1984torsion}}]
    Let $\torspair$ be a torsion pair in $\Modf(\Lambda)$ for a finite dimensional $\K$-algebra $\Lambda$. Then the following are equivalent.
    \begin{enumerate}
        \item $\tors$ is functorially finite,
        \item $\torsfree$ is functorially finite,
        \item $\tors=\gen(M)$ for some $M\in\Modf(\Lambda)$,
        \item $\torsfree=\sub(N)$ for some $N\in\Modf(\Lambda)$.
    \end{enumerate}
\end{thm}
In addition, we may note that any subcategory $\cX\subseteq \Modf(\Lambda)$ of finite type is functorially finite \cite[Prop.~4.2]{auslander1980preprojective}. The result above is extended in \cite{adachi_iyama_reiten_2014} to an equivalence between the set of functorially finite torsion pairs and the set of support $\tau$-tilting modules of $\Lambda$.

Another concept of interest for subcategories is that of universal extensions. In the derived category it may in fact be formulated as an approximation of shifted objects.
\begin{definition}
    Let $\cA$ be an abelian category and $\cX$ a class of objects in $\cA$. A short exact sequence 
    $\begin{tikzcd}[column sep=1.5em]
        X\arrow[tail]{r}&E\arrow[two heads]{r}&Y
    \end{tikzcd}$
    with $X\in\add\cX$ is a \emph{universal extension} from $Y$ to $\cX$ if it gives an epimorphism of functors $\eta_X\colon \begin{tikzcd}[ column sep=1.5em]\Hom_{\cA}(X,-)|_{\cX}\arrow[two heads]{r}& \Ext_{\cA}^1(Y,-)|_{\cX}\end{tikzcd}$.

    Should the morphism $\begin{tikzcd}[column sep=1.5em]
        E\arrow[two heads]{r}&Y
    \end{tikzcd}$ above be right minimal, we call the universal extension \emph{minimal}.
\end{definition}

\begin{remark}\label{remark:universal extension as pushout}
    We may also define a universal extension to $\cX$ through a push-out property. That is, a short exact sequence
    $
    \begin{tikzcd}[column sep=1.5em]
        X\arrow[tail]{r}&E\arrow[two heads]{r}&Y
    \end{tikzcd}
    $
    with $X\in \cX$ is a universal extension of $Y$ to $\cX$ if and only if for every other short exact sequence 
    $
    \begin{tikzcd}[column sep=1.5em]
        X'\arrow[tail]{r}&E'\arrow[two heads]{r}&Y
    \end{tikzcd}
    $
    with $X'\in \cX$, we can find morphisms such that the left hand square below is a pushout-square
    \[
    \begin{tikzcd}
        X\rar[tail]\arrow[dashed]{d}&E\rar[two heads]\arrow[dashed]{d}&Y\arrow[equal]{d}\\
        X'\rar[tail]&E'\rar[two heads]&Y
    \end{tikzcd}
    \]
\end{remark}

A universal extension, if it exists, need not be unique. However, a minimal universal extension is unique up to non-unique isomorphisms. Hence, it may be beneficial to work over an abelian category where arbitrary universal extensions can be made minimal. It turns out that the Krull--Schmidt property suffices. 
\begin{lemma}[{\cite[Cor.~1.4]{krause1997minimal}}]\label{Lemma:Decompose morphism in minimal}
    Let $\cA$ be an abelian Krull--Schmidt category. For any morphism $f\colon X\to Y$ in $\cA$ there exists a decomposition 
        \[
        f=\begin{pmatrix}
            f'&0
        \end{pmatrix}\colon 
        \begin{tikzcd}X=X'\oplus X''\rar& Y
        \end{tikzcd}
        \]
        such that $f'$ is right minimal.
\end{lemma}

Observe that this result characterizes minimal universal extensions to $\cX$ as universal extensions without summands on the form $
    \begin{tikzcd}[column sep=1.5em]
        X\arrow[tail]{r}&X\arrow[two heads]{r}&0
    \end{tikzcd}
    $ with $X\in \cX$. 
The connection of universal extensions and approximations can be made more explicit, as in the following lemma.

\begin{lemma}\label{Lemma: Covariantly finite gives universal}
    Let $\cA$ be an abelian category with enough projectives, and $\cX$ be a covariantly finite subcategory of $\cA$. Then every object $A$ in $\cA$ has a universal extension to $\cX$.
    \begin{proof}
    $\cA$ has enough projectives, so we construct from an epimorphism $P\to A$ of a projective object $P$ of $\cA$ to $A$, the short exact sequence
    \[
    \begin{tikzcd}[]
        \Omega A\arrow[tail]{r}&P\arrow[two heads]{r}&A
    \end{tikzcd}.
    \]
    This gives us the epimorphism $\begin{tikzcd}[column sep=1.5em]
        \Hom_\cA(\Omega A,-)\rar[two heads]&\Ext^1_\cA(A,-).
    \end{tikzcd}$ $\cX$ is covariantly finite, so we can find an epimorphism $\begin{tikzcd}[column sep=1.5em]\Hom_\cA(X,\,-)|_\cX\rar[two heads]&\Hom_\cA(\Omega A,\,-)|_\cX \end{tikzcd}$ with $X\in \cX$. By composition we now obtain an epimorphism 
    \[
    \begin{tikzcd}
        \Hom_\cA(X,-)|_\cX\rar[two heads]&\Ext_\cA^1(A,-)|_\cX.
    \end{tikzcd}
    \]
    By Yoneda Lemma we thus obtain our wanted extension.\qedhere
    \end{proof}
\end{lemma}

\introheaderTwo{Quotient categories.}
Before moving on from this section we also want to remind the reader of quotient categories (See e.g. \cite[Appendix~A.3]{assem2006elements}). Let $\cB$ be an additive category. A class of morphisms $\mathcal{I}$ in $\cB$ is an ideal if it satisfies the following.
\begin{itemize}
    \item for each $X\in \cB$, the zero morphism $0_X\in\Hom_\cB(X,\,X)$ belong to $\mathcal{I}$,
    \item if $f,g\colon X\to Y$ are both in $\mathcal{I}$, then $u\circ(f+g)\circ v\in\mathcal{I}$ for all morphisms $u\colon Y\to Z$ and $v\colon W\to X$.
\end{itemize}
Given an ideal $\mathcal{I}$ we let $\cB/\mathcal{I}$ denote the category with the same objects as $\cB$ and Hom-sets given by the quotient groups $\Hom_\cB(X,\,Y)/\mathcal{I}(X,\,Y)$. There is a canonical functor $\pi_\mathcal{I}\colon \cB\to \cB/\mathcal{I}$, acting as identity on objects and taking morphisms $f\in\Hom_\cB(X,\,Y)$ to their coset in $\Hom_\cB(X,\,Y)/\mathcal{I}(X,\,Y)$. In addition, this functor enjoy the property that any additive functor $F\colon \cB\to \cX$ such that $F(f)=0$ for all $f\in\mathcal{I}$ factors uniquely through $\pi_\mathcal{I}$.
\[
\begin{tikzcd}
    \cB\arrow{r}{F}\arrow[swap]{d}{\pi_\mathcal{I}}&\cX\\
    \cB/\mathcal{I}\arrow[dashed,swap]{ur}{\overline{F}}&
\end{tikzcd}
\]
If the class $\mathcal{I}$ is defined as all morphisms factoring through a class of objects $\cX$, we denote it as $[\cX]$.

\section{Main result}
In this section we will prove \Cref{thm: Main Theorem}. We start with a Wakamatsu-type lemma on universal extensions.

\begin{lemma}\label{Lemma:Wakamatsu-type}
    Let $\cX$ be a full additively closed subcategory of $\cA$ closed under extensions. If 
    $
    \begin{tikzcd}[column sep=1.5em]
        X\rar[tail]{f}&E\rar[two heads]{g}&Y
    \end{tikzcd}
    $
    is a minimal universal extension of $Y$ to $\cX$, then $E\in \extort\cX$.

    Moreover, $g$ is a right $\extort\cX$-approximation of $Y$.

    \begin{proof}
    Pick any 
    $
    \xi\colon
    \begin{tikzcd}[column sep=1.5em]
        X'\rar[tail]{}&U\rar[two heads]{h}&Y
    \end{tikzcd}
    \in\Ext^1_\cA(E,\,X')
    $
    
     with $X'\in \cX$. We can then construct the following pullback diagram
\[\begin{tikzcd}
	{X'} & {X'} \\
	L & U & Y \\
	X & E & Y
	\arrow[Rightarrow, no head, from=1-1, to=1-2]
	\arrow[tail, from=1-1, to=2-1]
	\arrow[tail, from=1-2, to=2-2]
	\arrow[tail, from=2-1, to=2-2]
	\arrow[two heads, from=2-1, to=3-1]
	\arrow["{g'}", two heads, from=2-2, to=2-3]
	\arrow["h", two heads, from=2-2, to=3-2]
	\arrow[Rightarrow, no head, from=2-3, to=3-3]
	\arrow[bend left, dashed, from=3-1, to=2-1]
	\arrow["f", tail, from=3-1, to=3-2]
	\arrow["{h'}", bend left, dashed, from=3-2, to=2-2]
	\arrow["g", two heads, from=3-2, to=3-3]
\end{tikzcd}\]
where $L\in\cX$ since $\cX$ is closed under extensions and the dashed arrows exist by the universality of the lower sequence. The morphism $g$ is right minimal and $g=g\circ h\circ h'$, so $h\circ h'$ is an automorphism and consequently $\xi$ splits.

The claim on $g$ being a right $\extort\cX$-approximation follows from the exact sequence
\[
\begin{tikzcd}[column sep=1.8em]
    \Hom_\cA(-,E)|_{\extort\cX}\rar{-\circ g}&\Hom_\cA(-,Y)|_{\extort\cX}\rar&\Ext_{\extort\cX}^1(-,X)|_{\extort\cX}=0
\end{tikzcd}
\]
associated to the universal extension.\qedhere
\end{proof}
\end{lemma}

For a torsion pair $\torspair$ and a torsion-free object with a universal extension to $\tors$, we therefore can associate to objects in $\extort\tors$, and as we will see in the following simple lemma, objects in $\extort\tors$ are associated to torsion-free objects with universal extensions.

\begin{lemma}\label{Lemma: extort sent to universal extensions}
    Let $\torspair$ be a torsion pair in $\cA$. The restriction of the torsion-free functor $\ffunc\colon \cA\to\torsfree$ to $\extort\tors$ has as essential image the subcategory of objects in $\torsfree$ that admit a universal extension to $\tors$.
    \begin{proof}
    Let $X\in \extort\tors$ be any object. From the torsion sequence of $X$,
    \[
    \begin{tikzcd}
        \tfunc X\rar[tail]&X\rar[two heads]&\ffunc X ,
    \end{tikzcd}
    \]
    we obtain the exact sequence
    \[
    \begin{tikzcd}
        \Hom_\cA(\tfunc X,-)|_\tors\rar&\Ext_\cA^1(\ffunc X,-)|_\tors\rar&\Ext_\cA^1(X,-)|_\tors=0.
    \end{tikzcd}
    \]
    Hence, $\ffunc X$ admits a universal extension to $\tors$. Conversely, from \Cref{Lemma:Wakamatsu-type} we see that for any minimal universal extension to $\tors$ of an object $F\in\torsfree$ we have an object $E\in \extort\tors$ such that $\ffunc E=F$, namely the middle term of the minimal universal extension. This completes the proof.\qedhere
\end{proof}
\end{lemma}

The two previous lemmas provide a way to associate objects in $\extort\tors$ to objects in $\torsfree$ admitting universal extensions to $\tors$, and vice versa. However, only the restriction of the torsion-free functor is functorial in general. Hence, we need to restrict appropriately for the other to be functorial. For convenience we will be denoting the full subcategory of torsion-free objects admitting a universal extension to $\tors$ by $\univTorsfree$. We also fix a minimal extension for each object $X\in\univTorsfree$ and call it the minimal universal extension of $X$.

\begin{lemma}
    \label{Lemma:functorial wakamatsu}
    Let $\torspair$ be a torsion pair in $\cA$, and 
    $
    \xi\colon 
    \begin{tikzcd}[column sep=1.5em]
        T\rar[tail]{}&E\rar[two heads]{}&F
    \end{tikzcd}
    $
    the minimal universal extension of $F\in \univTorsfree$. The assignment $F\mapsto E$ gives a functor $\cotorsfunc\colon \univTorsfree\to \extort\tors/\mathcal{I}$, where $\mathcal{I}$ is the ideal of all morphisms in $\extort\tors$ factoring through objects in $\tors$.
\begin{proof}
    Let $f\colon F\to F'$ be a morphism in $\univTorsfree$. We construct the commutative diagram
    \[
    \begin{tikzcd}
        T\arrow[tail]{r}\arrow[dashed]{d}{}&E\arrow[two heads]{r}\arrow[dashed]{d}{\cotorsfunc f}&F\arrow{d}{f}\\
        T'\arrow[tail]{r}&E'\arrow[two heads]{r}&F'
    \end{tikzcd}
    \]
    where the rows are minimal universal extensions. The long exact sequence induced by $\Hom_\cA(E,\,-)$ and the lower row guarantees the existence of $\cotorsfunc f\colon E\to E'$ making the right square commutative, since $\Ext_\cA(E,\,T')=0$ by \Cref{Lemma:Wakamatsu-type}. Finally for well-definedness, observe that the zero morphism $0\colon F\to F'$ is sent to a morphism $\cotorsfunc 0$ which factors through $T'$. Hence, we obtain a functor $\cotorsfunc\colon \univTorsfree\to \extort\tors/\mathcal{I}$.
\end{proof}
\end{lemma}

\begin{remark}\label{Remark: Middle term universal extension}
    We can think of the objects in $\extort\tors/\mathcal{I}$ as those in $\extort\tors$ with no summands in $\tors$. 
    Also, a universal extension $\epsilon\,\colon \begin{tikzcd}
        T\rar[tail]&E\rar[two heads]&F
    \end{tikzcd}$ from $F\in \torsfree$ to $\tors$ is minimal if and only if $E$ has no summand in $\tors$. One direction is proven contrapositively. Let
    \[
    \epsilon\,\colon 
    \begin{tikzcd}[]
        T\rar[tail]{}&E\rar[two heads]{}&F
    \end{tikzcd}
    \]
    be a universal extension from $F$ to $\tors$, and assume that $E\cong E'\oplus T'$ with $T'\in \tors$. Then $\epsilon$ is isomorphic to
    \[
    \begin{tikzcd}[ampersand replacement=\&]
        T''\oplus T'\rar[tail]{}\&E'\oplus T'\rar[two heads]{\begin{psmallmatrix}f&0\end{psmallmatrix}}\&F
    \end{tikzcd}
    \]
    for some morphism $f\colon E'\to F$. This contradicts the minimality of $\epsilon$. The other direction is shown similarly using the decomposition in \Cref{Lemma:Decompose morphism in minimal} and the fact that $\tors$ is additively closed.
\end{remark}

We are now ready to prove \Cref{thm: Main Theorem}.

\begin{thm}[{\Cref{thm: Main Theorem}}]\label{Theorem: Main result}
    Let $\cA$ be a Krull--Schmidt abelian category and let $\torspair$ be a torsion pair in $\cA$. The torsion-free functor $\ffunc\colon \cA\to \torsfree$ induce an equivalence
    \begin{equation}\label{eq:equivalence main result in text}
    \frac{\extort\tors}{\mathcal{I}}\simeq \univTorsfree
    \end{equation}
    where $\mathcal{I}$ is the ideal in $\extort\tors$ of morphisms factoring through $\tors$.
\end{thm}
\begin{proof}
    We observe that since $\ffunc(\tors)=0$, we obtain a functor $\overline{\ffunc}\colon \cA/[\tors]\to\torsfree$ making the following diagram commute.
    \[
    \begin{tikzcd}
        \cA\arrow{r}{\ffunc}\arrow[swap,two heads]{d}{\pi_{[\tors]}}&\torsfree\\
        \cA/[\tors]\arrow[dashed,swap]{ur}{\overline{\ffunc}}
    \end{tikzcd}
    \]
    The restriction of $\pi_{[\tors]}$ to $\extort\tors$ has $\extort\tors/\mathcal{I}$ as essential image, so by appropriate restrictions we obtain the following commutative diagram,
    \[
    \begin{tikzcd}
        \extort\tors\arrow{r}{\ffunc}\arrow[two heads,swap]{d}{\pi_\mathcal{I}}&\torsfree\\
        \extort\tors/\mathcal{I}\arrow[dashed]{ur}{\overline{\ffunc}}\arrow[dotted,two heads]{r}{\mathfrak{F}}&\univTorsfree\arrow[tail,dotted,swap]{u}{\mathrm{inc}}
    \end{tikzcd}
    \]
    where the factorization of $\overline{\ffunc}$ through $\univTorsfree$ arise from \Cref{Lemma: extort sent to universal extensions}.
    We now claim that $\mathfrak{F}$ and $\cotorsfunc\colon \univTorsfree\to\extort\tors/\mathcal{I}$ from \Cref{Lemma:functorial wakamatsu} are quasi-inverse functors. 
    
    \textbf{\underline{$\mathfrak{F}\circ \cotorsfunc\cong \mathrm{id}_\univTorsfree$:}} We start by recalling that for a short exact sequence 
    $
    \begin{tikzcd}[column sep=1.5em]
        T\rar[tail]&E\rar[two heads]&F
    \end{tikzcd}
    $
    with $T\in \tors$ and $F\in \torsfree$, we have $F\cong \ffunc E$. Let now $f\colon F\to F'$ be a morphism in $\univTorsfree$, then we have as in the proof of \Cref{Lemma:functorial wakamatsu} a commutative diagram
    \[
    \begin{tikzcd}
        T\arrow[tail]{r}\arrow[dashed]{d}{\iota}&E\arrow[two heads]{r}\arrow[dashed]{d}{\cotorsfunc f}&F\arrow{d}{f}\\
        T'\arrow[tail]{r}&E'\arrow[two heads]{r}&F'
    \end{tikzcd}
    \]
    where the rows are torsion sequences of $E$ and $E'$ respectively. As $\ffunc\colon \cA\to \torsfree$ sends morphisms uniquely, we then see that $\mathfrak{F}(\cotorsfunc f)=f$, and hence $\mathfrak{F}\circ \cotorsfunc\cong \mathrm{id}_\univTorsfree$.

    \textbf{\underline{$\cotorsfunc\circ\mathfrak{F}\cong \mathrm{id}_{\extort\tors/\mathcal{I}}$:}} Let $X$ be an object in $\extort\tors$ without summands in $\tors$. \Cref{Lemma: extort sent to universal extensions} and \Cref{Remark: Middle term universal extension} tells us that the torsion sequence
    \[
    \begin{tikzcd}
        \tfunc X\rar[tail] &X\rar[two heads]&\ffunc X
    \end{tikzcd}
    \]
    is a minimal universal extension. It follows from minimal universal extensions being unique up to (non-unique) isomorphisms, that we may construct the following commutative diagram
    \[
    \begin{tikzcd}
        \tfunc X\rar[tail]\dar{\cong} &X\rar[two heads]\arrow["\cong","\varepsilon_X"']{d}& \ffunc X\dar[equal]\\
        T\rar[tail]& E\rar[two heads]&\ffunc X
    \end{tikzcd}
    \]
    where the lower row is the fixed universal extension of $\ffunc E$, showing that $\cotorsfunc\circ \mathfrak{F}(X)\cong X$.

    Let now $f\in\Hom_{\extort\tors/\mathcal{I}}(X,X')$ be any morphism in $\extort\tors/\mathcal{I}$, and $g\in \Hom_{\extort\tors}(X,X')$ in the preimage of $f$ under $\pi_\mathcal{I}$. We then construct the following commutative diagram where the two rows in the middle are torsion sequences and the upper (lower) row are the minimal universal extension of $\ffunc X$ ($\ffunc X'$) into $\tors$:
    \[
    \begin{tikzcd}
	T & E & {\ffunc X} \\
	{\tfunc X} & X & {\ffunc X} \\
	{\tfunc X'} & {X'} & {\ffunc X'} \\
	{T'} & {E'} & {\ffunc X'}
	\arrow[tail,from=1-1, to=1-2]
	\arrow["\cong", from=1-1, to=2-1]
	\arrow[two heads,from=1-2, to=1-3]
	\arrow["\cong","\varepsilon_{X}"', from=1-2, to=2-2]
	\arrow[equal,from=1-3, to=2-3]
	\arrow[tail,from=2-1, to=2-2]
	\arrow["{\tfunc g}",swap, from=2-1, to=3-1]
	\arrow[two heads,from=2-2, to=2-3]
	\arrow["g", from=2-2, to=3-2]
	\arrow["{\ffunc g}", from=2-3, to=3-3]
	\arrow[tail,from=3-1, to=3-2]
	\arrow["\cong", from=3-1, to=4-1]
	\arrow[two heads,from=3-2, to=3-3]
	\arrow["\cong"',"\varepsilon_{X'}", from=4-2, to=3-2]
	\arrow[equal,from=3-3, to=4-3]
	\arrow[tail,from=4-1, to=4-2]
	\arrow[two heads,from=4-2, to=4-3]
    \end{tikzcd}
    \]
    The well-definedness of $\cotorsfunc$ now tells us that 
    $$
    f\cong\pi_\mathcal{I}(\varepsilon_X\circ g\circ \varepsilon_{X'}^{-1})=\cotorsfunc\pi_\mathcal{I}(\ffunc g)=\cotorsfunc\mathfrak{F}f,
    $$
    finishing the proof.
\end{proof}

\section{When all objects admit universal extensions}
The motivation for the article as stated in the introduction, was to generalize Corollary~4.8 in \cite{Buan2021tauTilting} by removing the assumption on functorially finiteness of the torsion pairs. We will now verify that our results do in fact specialize back when the torsion pairs are functorially finite. We therefore show the following corollary of \Cref{Theorem: Main result}, whose assumptions in particular hold when the torsion pair is functorially finite and $\cA$ has enough projectives, see \Cref{Lemma: Covariantly finite gives universal}.

\begin{corol}\label{Lem:func fin tors give old equivalence}
    Let $\cA$ be a Krull--Schmidt abelian category and let $\torspair$ be a torsion pair in $\cA$. If all objects in $\tors$ and $\torsfree$ admit a universal extension to $\tors$, then we have an equivalence
    \begin{equation}\label{eq:equivalence to generalize in text}
    \frac{\extort\tors}{\extort\tors\cap\tors}\simeq \torsfree.
    \end{equation}

    \begin{proof}
        The right hand side of \eqref{eq:equivalence main result in text} clearly coincide with that of \eqref{eq:equivalence to generalize in text} when all objects in $\torsfree$ have universal extensions to $\tors$. We are thus only concerned with showing that the left hand sides also coincide. 
        
        Let $f\in\Hom_\cA(X,X')$ be a morphism in $\extort\tors$ factoring through $\tors$, specifically we can find morphisms $\phi\in\Hom_\cA(X,Y)$ and $\psi\in\Hom_\cA(Y,X')$ for $Y\in\tors$ such that $f=\psi\circ \phi$. Further, let 
        $\begin{tikzcd}[column sep=1.5em]
            T\rar[tail]&E\rar[two heads]{g}&Y
        \end{tikzcd}$
        be a minimal universal extension of $Y$ to $\tors$. From this we construct the following commutative diagram
        \[
        \begin{tikzcd}
            &&X\arrow[dashed]{dl}\dar{\phi}\\
            T\rar[tail]&E\rar[two heads]{g}&Y
        \end{tikzcd}
        \]
        where the dashed arrow exists since $g$ is a right $\extort\tors$ approximation of $Y$ as observed in \Cref{Lemma:Wakamatsu-type}. Now, since $\tors$ is closed under extensions we see that $E\in \extort\tors\cap\tors$. Hence, every morphism from $\extort\tors$ to $\tors$ factors through $\extort\tors\cap\tors$, and the proof is done. \qedhere
\end{proof}
\end{corol}

\begin{example}\label{ex:torsion pair A3}
    Consider the path algebra of .
    \[
    \begin{tikzpicture}
        \node at (0,0) []{$\overset{\rightarrow}{\qA}_3\colon$};

        \node (3) at (1,0) []{$\bullet$};
        \node (2) at (2,0) []{$\bullet$};
        \node (1) at (3,0) []{$\bullet$};
        \draw[-latex] (3.east)--(2.west);
        \draw[-latex] (2.east)--(1.west);
        \node at (3) [below,scale=.8]{$3$};
        \node at (2) [below,scale=.8]{$2$};
        \node at (1) [below,scale=.8]{$1$};
    \end{tikzpicture}.
    \]
    Let $\torsfree$ be the subcategory of $\Modf\K\qA_3$ additively generated by all the projectives and the injective at $3$, and $\tors$ the subcategory additively generated by the simple at $2$. In the AR-quiver of $\K\qA_3$ below we have marked the indecomposable in $\tors$ by a square, the indecomposables in $\torsfree$ as filled circles and the indecomposable in neither by a non-filled circle.
    
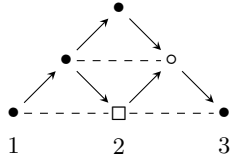
\[
    \begin{tikzpicture}[scale=.7,tips=proper]
        \coordinate (S1) at (0,0);
        \coordinate (S2) at (2,0);
        \coordinate (S3) at (4,0);
        \coordinate (P2) at (1,1);
        \coordinate (P3) at (2,2);
        \coordinate (I2) at (3,1);
        
        \node[label={below:\small $1$}] at (S1)  {$\bullet$};
        \node[label={below:\small $2$}] at (S2) {$\square$};
        \node[label={below:\small $3$}] at (S3) {$\bullet$};
        \node at (P2) [] {$\bullet$};
        \node at (I2) [] {$\circ$};
        \node at (P3) [] {$\bullet$};

        \draw[-stealth] 
            {($(S1.center)+(.2,.2)$) edge ($(P2.center)-(.2,.2)$)}
            {($(P2.center)+(.2,.2)$) edge ($(P3.center)-(.2,.2)$)}
            {($(P2.center)+(.2,-.2)$) edge ($(S2.center)-(.2,-.2)$)}
            {($(S2.center)+(.2,.2)$) edge ($(I2.center)-(.2,.2)$)}
            {($(P3.center)+(.2,-.2)$) edge ($(I2.center)-(.2,-.2)$)}
            {($(I2.center)+(.2,-.2)$) edge ($(S3.center)-(.2,-.2)$)};
        \draw[dashed]
            {($(S1.center)+(.2,0)$) edge ($(S2.center)-(.2,0)$)} 
            {($(S2.center)+(.2,0)$) edge ($(S3.center)-(.2,0)$)} 
            {($(P2.center)+(.2,0)$) edge ($(I2.center)-(.2,0)$)} ;
    \end{tikzpicture}
    \]
    We have illustrated the equivalence in \eqref{eq:equivalence to generalize in text} as follows. In the left hand side picture we have shaded the indecomposables in $\extort\tors\setminus \tors$ and in the right hand side picture we have shaded $\univTorsfree=\torsfree$.
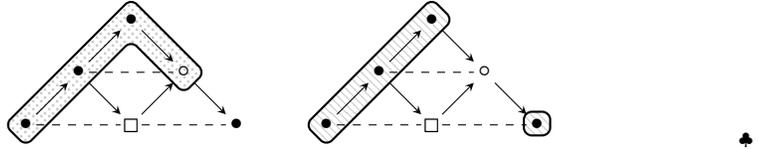
    \[
    \begin{tikzpicture}[scale=.7,tips=proper]
        \coordinate (S1) at (0,0);
        \coordinate (S2) at (2,0);
        \coordinate (S3) at (4,0);
        \coordinate (P2) at (1,1);
        \coordinate (P3) at (2,2);
        \coordinate (I2) at (3,1);
        
        \node[] at (S1)  {$\bullet$};
        \node[] at (S2) {$\square$};
        \node[] at (S3) {$\bullet$};
        \node at (P2) [] {$\bullet$};
        \node at (I2) [] {$\circ$};
        \node at (P3) [] {$\bullet$};

        \draw[-stealth] 
            {($(S1.center)+(.2,.2)$) edge ($(P2.center)-(.2,.2)$)}
            {($(P2.center)+(.2,.2)$) edge ($(P3.center)-(.2,.2)$)}
            {($(P2.center)+(.2,-.2)$) edge ($(S2.center)-(.2,-.2)$)}
            {($(S2.center)+(.2,.2)$) edge ($(I2.center)-(.2,.2)$)}
            {($(P3.center)+(.2,-.2)$) edge ($(I2.center)-(.2,-.2)$)}
            {($(I2.center)+(.2,-.2)$) edge ($(S3.center)-(.2,-.2)$)};
        \draw[dashed]
            {($(S1.center)+(.2,0)$) edge ($(S2.center)-(.2,0)$)} 
            {($(S2.center)+(.2,0)$) edge ($(S3.center)-(.2,0)$)} 
            {($(P2.center)+(.2,0)$) edge ($(I2.center)-(.2,0)$)} ;

        \begin{scope}[on background layer]
            \fill[pattern=crosshatch dots,pattern color=gray!40,rounded corners=3pt] ($(P3.center)+(0,.4)$)--($(S1.center)-(.4,0)$)--($(S1.center)+(0,-.4)$)--($(P3.center)-(0,0.4)$)--($(I2.center)-(0,.4)$)--($(I2.center)+(0.4,0)$)--cycle;
        
        \end{scope}
        \draw[thick,rounded corners=3pt] ($(P3.center)+(0,.4)$)--($(S1.center)-(.4,0)$)--($(S1.center)+(0,-.4)$)--($(P3.center)-(0,0.4)$)--($(I2.center)-(0,.4)$)--($(I2.center)+(0.4,0)$)--cycle;
    \end{tikzpicture}
    \qquad
    \begin{tikzpicture}[scale=.7,tips=proper]
        \coordinate (S1) at (0,0);
        \coordinate (S2) at (2,0);
        \coordinate (S3) at (4,0);
        \coordinate (P2) at (1,1);
        \coordinate (P3) at (2,2);
        \coordinate (I2) at (3,1);
        
        \node[] at (S1)  {$\bullet$};
        \node[] at (S2) {$\square$};
        \node[] at (S3) {$\bullet$};
        \node at (P2) [] {$\bullet$};
        \node at (I2) [] {$\circ$};
        \node at (P3) [] {$\bullet$};

        \draw[-stealth] 
            {($(S1.center)+(.2,.2)$) edge ($(P2.center)-(.2,.2)$)}
            {($(P2.center)+(.2,.2)$) edge ($(P3.center)-(.2,.2)$)}
            {($(P2.center)+(.2,-.2)$) edge ($(S2.center)-(.2,-.2)$)}
            {($(S2.center)+(.2,.2)$) edge ($(I2.center)-(.2,.2)$)}
            {($(P3.center)+(.2,-.2)$) edge ($(I2.center)-(.2,-.2)$)}
            {($(I2.center)+(.2,-.2)$) edge ($(S3.center)-(.2,-.2)$)};
        \draw[dashed]
            {($(S1.center)+(.2,0)$) edge ($(S2.center)-(.2,0)$)} 
            {($(S2.center)+(.2,0)$) edge ($(S3.center)-(.2,0)$)} 
            {($(P2.center)+(.2,0)$) edge ($(I2.center)-(.2,0)$)} ;

        \begin{scope}[on background layer]
        \fill[pattern=north west lines,pattern color=gray!40,rounded corners=3pt] ($(S3.center)-(0.25,.2)$)--($(S3.center)+(0.25,-.2)$)--($(S3.center)+(0.25,.25)$)--
        ($(S3.center)+(-0.25,.25)$)--cycle;
        \fill[pattern=north west lines,pattern color=gray!40,rounded corners=3pt] ($(P3.center)+(0,.4)$)--($(S1.center)-(.4,0)$)--($(S1.center)+(0,-.4)$)--($(P3.center)+(0.4,0)$)--cycle;
        
    \end{scope}
    
    \draw[thick,rounded corners=3pt] ($(S3.center)-(0.25,.2)$)--($(S3.center)+(0.25,-.2)$)--($(S3.center)+(0.25,.25)$)--
        ($(S3.center)+(-0.25,.25)$)--cycle;
    \draw[thick,rounded corners=3pt] ($(P3.center)+(0,.4)$)--($(S1.center)-(.4,0)$)--($(S1.center)+(0,-.4)$)--($(P3.center)+(0.4,0)$)--cycle;
    \end{tikzpicture}\qedhere
    \]
\end{example}

In the next section we will consider Krull-Schmidt categories without enough projectives. However, using a similar construction as in the proof of Bongartz' lemma \cite[Lemma~2.1]{Bon81}, we can see that since they are Ext-finite (all Ext-groups have finite dimension), certain torsion classes still satisfy the condition of \Cref{Lem:func fin tors give old equivalence}. The following is a sketch of the construction.

Let $\cA$ be an abelian $\K$-category, and for an object $X\in\cA$ let $\cX=\add(X)$. Assume now that for $Y\in \cA$ we have $d=\mathrm{dim}\Ext_\cA^1(Y,X)<\infty$. Let 
$\mathbb{E}_i\colon\begin{tikzcd}[column sep=1.5em]
    X\rar[tail]&E_i\rar[two heads]&Y
\end{tikzcd}$ for $1\leq i\leq d$ be representatives of the basis elements of $\Ext_\cA(Y,X)$. 
We then construct the following pullback diagram
\[\tikzset{pullback/.style={minimum size=1.2ex,path picture={
\draw[opacity=1,black,-] (0.5ex,-0.5ex) -- (0.5ex,0.5ex) -- (-0.5ex,0.5ex);%
}}}
\begin{tikzcd}
    X^d\arrow[tail]{r}&\bigoplus_{i=1}^d E_i\arrow[two heads]{r}&Y^d\\
    X^d\arrow[equal]{u}\arrow[tail]{r}&E\arrow{u}\arrow[ur, phantom," " {pullback}, very near start]\arrow[two heads]{r}&Y\arrow[swap]{u}{\begin{bsmallmatrix}
        1\\1\\\vdots\\1
    \end{bsmallmatrix}}
\end{tikzcd}
\]
where the upper row is the direct sum of $\mathbb{E}_i$ for $1\leq i\leq d$. The lower row can be shown to be a universal extension from $Y$ to $\cX$. Thus we have the following lemma.

\begin{lemma}\label{lemma: ext finite add gen tors enough}
    Let $\cA$ be an Ext-finite abelian $\K$-category and $\cX=\add(X)$ a subcategory. Then every object $Y\in\cA$ admit a universal extension to $\cX$.
\end{lemma}

Combining \Cref{Lem:func fin tors give old equivalence} with \Cref{lemma: ext finite add gen tors enough} then gives us the following.

\begin{corol}\label{cor:equivalence when ext-finite}
    Let $\cA$ be an Ext-finite Krull-Schmidt abelian $\K$-category. If $\torspair$ is a torsion pair such that $\tors=\add(T)$ for some object $T\in \tors$, then we have an equivalence
    \begin{equation*}\label{eq:equivalence when ext-finite}
    \frac{\extort\tors}{\extort\tors\cap\tors}\simeq \torsfree.
    \end{equation*}
\end{corol}

\begin{remark}
    In \cite[Def.~0.2]{Buan2021tauTilting} the concept of \emph{left-weak cotorsion torsion triples} is introduced as a way to generalize the cotorsion torsion triples used in \cite{bauer2020cotorsion}. It is shown that these arise in connection with every functorially finite torsion pair for $\cA=\Modf\Lambda$, with $\Lambda$ a finite dimensional algebra. This can effortlessly and non-surprisingly be shown to also hold in a Krull--Schmidt abelian category with enough projectives. Let us briefly explain why. 
    
    In order for a pair $(\cC,\tors)$ of subcategories to be a left-weak cotorsion pair it need to satisfy the following
    \begin{enumerate}
        \item $\Ext_\cA^1(\cC,\tors)=0$, and
        \item for every $A\in \cA$ there exists
        \begin{enumerate}
            \item an exact sequence
            \[
                \begin{tikzcd}
                    T\rar[tail]&E\rar[two heads]{f}&A
                \end{tikzcd}
            \]
            with $T\in\tors$ and $f$ a right $\cC$-approximation of $A$, and
            \item an exact sequence
            \[
                \begin{tikzcd}
                    A\rar{g}&T'\rar[two heads]&X
                \end{tikzcd}
            \]
        such that $g$ is a left $\tors$-approximation of $A$ and $X\in \cC$. Note, $g$ is not necessarily a monomorphism.
        \end{enumerate}
    \end{enumerate}
    We have by construction that $(\extort\tors,\tors)$ satisfies (1), and (2)(a) follows from \Cref{Lemma: Covariantly finite gives universal} and \Cref{Lemma:Wakamatsu-type}. Finally, (2)(b) follows from the covariantly finiteness of $\tors$ and Wakamatsu's Lemma, see e.g. \cite[Lem.~2.1.13]{gobel2006approximations}.
\end{remark}

\section{Tame Hereditary Algebras}
Let us now investigate our equivalence for non-functorially finite torsion classes. In order for these to exist we need to move beyond representation-finite algebras. The simplest examples are tame hereditary algebras, hence we assume that $\Lambda$ is a path algebra of Euclidean type. Unless otherwise specified we assume our torsion pairs $\torspair$ to be non-functorially finite from now.

The AR-quivers of such algebras are well understood. They consist of a preprojective component $\preproj$, a preinjective component $\preinj$ and a family of regular components $\regular$. 
Further, the preprojective component is isomorphic to $(-\mathbb{N})Q^{\text{op}}$ and the preinjective component is isomorphic to $\mathbb{N}Q^{\text{op}}$. We refer to \cite{assem2006elements,daniel_elements_2007,Rin84} for more background on the structure of these components.  

In this section, we will focus on the regular component since by the following lemma both the preprojective and preinjective component behaves the same for all non-functorially finite torsion pairs. The result is well known, but for the convenience of the reader we write out the proof.
\begin{lemma}\label{lemma:non functorially finite only interesting in regular}
    Let $\Lambda$ be a tame hereditary algebra. If $\torspair$ is a non-functorially finite torsion pair, then $\preproj\subseteq  \torsfree$ and  $\preinj\subseteq\tors$
\end{lemma}
\begin{proof}
    We only show the inclusion of the preprojectives in the torsion free subcategory. Assume with a goal of contradiction that there is some $M=\tau^{-t}P_i\in\preproj\setminus\torsfree$. Then by the canonical torsion sequence and the fact that $\Hom_\Lambda(\preinj,\,\preproj)=\Hom_\Lambda(\regular,\,\preproj)=0$, we have that $\tfunc M\in \preproj$.
    Without loss of generality we therefore assume $M\in \tors$. 

    Any module without support at the simple $S_i$ is also a module over the corresponding quotient algebra $\Lambda/\langle e_i\rangle$, which is of finite representation type. Hence there are only finitely many indecomposable modules $N$ in $\Modf(\Lambda)$ such that $\Hom_\Lambda(P_i,\,N)=0$. Now, for any module $X$ in $\Modf(\Lambda)$, we have
    \[
    \Hom_\Lambda(M,\,X)=\Hom_\Lambda(\tau^{-t}P_i,\,X)\cong \Hom_\Lambda(P_i,\,\tau^t X).
    \]
    All these hom-sets vanish if $X$ is torsion-free. We can therefore conclude that $\torsfree$ may only contain finitely many indecomposables, which by \cite[Prop.~4.2]{auslander1980preprojective} means it is functorially finite. This contradiction finishes the proof.
\end{proof}

The regular components $\regular$ are hom-orthogonal tubes $\tube_{\lambda\in \mathbb{P}^1(\K)}$, each a serial Krull--Schmidt abelian category. Their AR-quivers are isomorphic to the translation quiver $\mathbb{Z}\qA_\infty/(\tau^r)$ where $r\geq 1$ is called the rank of $\tube_{\lambda}$. When $r=1$ the tube is called homogeneous. There are at most three non-homogeneous tubes in $\regular$. A consequence of the tubes being hom-orthogonal is that $\torspair$ is a non-functorially finite torsion pair of $\Modf(\Lambda)$ only if $(\tors\cap\tube,\,\torsfree\cap\tube)$ is a torsion pair for every tube $\tube$ and $\preproj\subseteq\torsfree$, $\preinj\subseteq\tors$.

We have collected a few elementary observations in the following lemma.

\begin{lemma}
    Let $\torspair$ be a non-functorially finite torsion pair in $\Modf(\Lambda)$.
    \begin{enumerate}
        \item For a homogeneous tube $\tube$, either $\tube\subseteq\tors$ or $\tube\subseteq\torsfree$,
    \end{enumerate} 
    Let $\cC$ be an AR-component.
    \begin{enumerate}
        \item[(2)] If $\cC\subseteq \torsfree$, then there are no non-trivial extensions from $\cC$ to $\tors$. Specifically, within $\cC$ we have that both sides of the equivalence \eqref{Theorem: Main result} are equal to $\cC$, with the universal extensions being trivial.
        \item[(3)] If $\cC\subseteq \tors$, then both sides of the equivalence \eqref{Theorem: Main result}  are equal to zero.
    \end{enumerate}
\end{lemma}

In light of this, we may restrict ourselves further to considering torsion pairs $\torspair$ such that $\tors\cap\tube\neq 0\neq \torsfree\cap\tube$ for at least one non-homogeneous tube $\tube$. Further since there are no non-trivial maps between tubes and from the preinjective component to a tube, any universal extension from $F\in\tube$ to $\tors$ is also a universal extension to $\tors\cap \tube$. Hence, we may look for universal extension only locally within tubes.

\subsection{Torsion pairs in tubes}\label{sec:torsion pairs in tubes}
Torsion pairs in tubes have been classified in \cite{baur2014torsion} through a bijection with cotilting modules. In order to utilize their result we need to establish a bit of theory and terminology; for the details we refer back to \cite{baur2014torsion}. In the following we assume that $\tube$ is a tube of rank $r> 1$. From now, unless specified, a pair $\torspair$ is a torsion pair with respect to the tube $\tube$.

We denote the simples in $\tube$ by $\ind{i}{i}$ for $i \in \mathbb{Z}$, such that $\tau \ind{i}{i}=\ind{i-1}{i-1}$ and $\ind{i+r}{i+r}=\ind{i}{i}$. The indecomposables are uniquely given through their composition series since $\tube$ is serial. An indecomposable of length $l$ and socle $\ind{i}{i}$ may therefore unambiguously be denoted $\ind{i}{i+l-1}$. With this convention we have the almost split sequences
\[
\begin{tikzcd}
    \ind{i-1}{i+l-2}\rar[tail]& \ind{i-1}{i+l-1}\oplus \ind{i}{i+l-2}\rar[two heads]& \ind{i}{i+l-1}
\end{tikzcd}
\]
where $\ind{i}{i+l-2}=0$ if $l<2$. The AR-quiver takes the shape

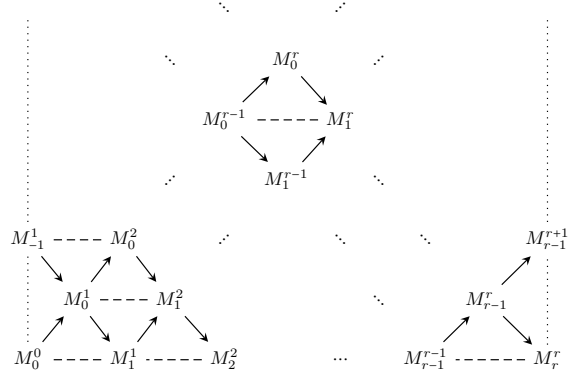
\[\begin{tikzcd}[column sep=.8,row sep=10,cells={nodes={scale=.7}}]
	{} &&&& \ddots && \udots &&&& {} \\
	&&& \ddots && {\ind{0}{r}} && \udots \\
	&&&& {\ind{0}{r-1}} && {\ind{1}{r}} \\
	&&& \udots && {\ind{1}{r-1}} && \ddots \\
	{ \ind{-1}{1}} && {\ind{0}{2}} && \udots && \ddots && \ddots && {\ind{r-1}{r+1}} \\
	& {\ind{0}{1}} && {\ind{1}{2}} &&&& \ddots && {\ind{r-1}{r}} \\
	{\ind{0}{0}} && {\ind{1}{1}} && {\ind{2}{2}} && \cdots && {\ind{r-1}{r-1}} && {\ind{r}{r}}
	\arrow[from=2-6, to=3-7]
	\arrow[from=3-5, to=2-6]
	\arrow[from=3-5, to=4-6]
	\arrow[dashed, no head, from=3-7, to=3-5]
	\arrow[from=4-6, to=3-7]
	\arrow[dotted, no head, from=5-1, to=1-1]
	\arrow[from=5-1, to=6-2]
	\arrow[dashed, no head, from=5-3, to=5-1]
	\arrow[from=5-3, to=6-4]
	\arrow[dotted, no head, from=5-11, to=1-11]
	\arrow[from=6-2, to=5-3]
	\arrow[from=6-2, to=7-3]
	\arrow[dashed, no head, from=6-4, to=6-2]
	\arrow[from=6-4, to=7-5]
	\arrow[from=6-10, to=5-11]
	\arrow[from=6-10, to=7-11]
	\arrow[dotted, no head, from=7-1, to=5-1]
	\arrow[from=7-1, to=6-2]
	\arrow[from=7-3, to=6-4]
	\arrow[dashed, no head, from=7-3, to=7-1]
	\arrow[dashed, no head, from=7-5, to=7-3]
	\arrow[from=7-9, to=6-10]
	\arrow[dotted, no head, from=7-11, to=5-11]
	\arrow[dashed, no head, from=7-11, to=7-9]
\end{tikzcd}\]
with the left and right hand sides identified. The tube $\tube$ as defined here does not contain any projectives nor any injectives, hence when we consider Ext-projectives (Ext-injectives) of a subcategory in $\tube$ and the context is clear, we drop the prefix.

To each simple object $\ind{i}{i}$ we define the following two additive subcategories:
\begin{enumerate}
    \item The \emph{ray} at $i$, denoted $\ray{i}$, is additively generated by the indecomposables having $\ind{i}{i}$ as socle. Equivalently,
    \[
    \ray{i}=\add\{\, \ind{i}{i},\  \ind{i}{i+1},\  \cdots,\ \ind{i}{i+l-1},\  \cdots\,\}.
    \]
    \item The \emph{coray} at $i$, denoted $\coray{i}$, is additively generated by the indecomposables having $\ind{i}{i}$ as top. Equivalently,
    \[
    \coray{i}=\add\{\, \cdots,\  \ind{i-l+1}{i},\ \cdots,\  \ind{i-1}{i},\  \ind{i}{i}  \,\}.
    \]
\end{enumerate}

For $1\leq t\leq r$ and $i\in\mathbb{Z}$ we define an additive category $\wing{i}{j}$ called a wing, additively generated by the subfactors of $\ind{i}{j}$. We could also describe it as $\ind{a}{b}$ being in $\wing{i}{j}$ if $i\leq a\leq b\leq j$. 
To make the definitions clearer, observe the following tube of rank $5$ with the coray at $1$ and the wing $\wing{2}{4}$ drawn in. 
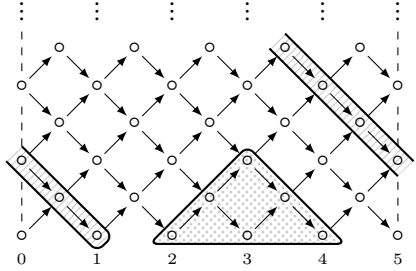
\[
\begin{tikzpicture}
    
    \foreach \x in {0,1,...,5}{
    \node (S\x) at (\x,0)[]{$\circ$};
    \node (r3-\x) at (\x,1)[]{$\circ$};
    \node (r5-\x) at (\x,2)[]{$\circ$};
    \node (r7-\x) at (\x,3)[]{$\vdots$};
    }
    \foreach \x in {1,2,3,4,5}{
    \node (r2-\x) at ({\x-.5},.5) []{$\circ$};
    \node (r4-\x) at ({\x-.5},1.5) []{$\circ$};
    \node (r6-\x) at ({\x-.5},2.5) []{$\circ$};
    }
    \foreach \x in {0,1,...,4}{
    \node at (S\x) [below,yshift=-3pt]{\tiny$\x$};
    }
    \node at (S5) [below,yshift=-3pt]{\tiny$5$};
    \draw[dashed] (S0)--(r3-0);
    \draw[dashed] (r3-0)--(r5-0);
    \draw[dashed] (r5-0)--(r7-0);
    \draw[dashed] (S5)--(r3-5);
    \draw[dashed] (r3-5)--(r5-5);
    \draw[dashed] (r5-5)--(r7-5);
    \foreach \x [evaluate=\x as \xx using int(\x+1)] in {0,1,2,3,4}{
        \draw[-latex] ($(S\x)+(.1,.1)$)--($(r2-\xx)-(.1,.1)$);
        \draw[-latex] ($(r3-\x)+(.1,.1)$)--($(r4-\xx)-(.1,.1)$);
        \draw[-latex] ($(r3-\x)+(.1,-.1)$)--($(r2-\xx)+(-.1,.1)$);
        \draw[-latex] ($(r5-\x)+(.1,-.1)$)--($(r4-\xx)+(-.1,.1)$);
        \draw[-latex] ($(r5-\x)+(.1,.1)$)--($(r6-\xx)-(.1,.1)$);

        \draw[-latex] ($(r2-\xx)+(.1,-.1)$)--($(S\xx)+(-.1,.1)$);
        \draw[-latex] ($(r2-\xx)+(.1,.1)$)--($(r3-\xx)-(.1,.1)$);
        \draw[-latex] ($(r4-\xx)+(.1,-.1)$)--($(r3-\xx)+(-.1,.1)$);
        \draw[-latex] ($(r4-\xx)+(.1,.1)$)--($(r5-\xx)-(.1,.1)$);
        \draw[-latex] ($(r6-\xx)+(.1,-.1)$)--($(r5-\xx)+(-.1,.1)$);
    }
    \begin{scope}[on background layer]
        \fill[pattern=crosshatch dots,pattern color=gray!40,rounded corners=3pt] ($(S2)-(.3,.1)$)--(r3-3.north)--($(S4)+(.3,-.1)$)--cycle;
        \fill[pattern=north east lines, pattern color=gray!40,rounded corners=3pt] (r3-0.west)--(S1.south)--(S1.east)--(r3-0.north);
        \fill[pattern=north east lines, pattern color=gray!40] (r3-5.south)--(r6-4.west)--(r6-4.north)--(r3-5.east)--cycle;
    \end{scope}
    \draw[thick,rounded corners=3pt] ($(S2)-(.3,.1)$)--(r3-3.north)--($(S4)+(.3,-.1)$)--cycle;
    \draw[thick,rounded corners=3pt] (r3-0.west)--(S1.south)--(S1.east)--(r3-0.north);
    \draw[thick] (r3-5.south)--(r6-4.west);
    \draw[thick] (r3-5.east)--(r6-4.north);
    
\end{tikzpicture}
\]

We are now splitting the rest of our investigation in two cases since we by Corollary~5.5 in \cite{baur2014torsion} know that either 
\begin{enumerate}
    \item $\tors$ is of finite type and $\torsfree$ of infinite type, or
    \item $\tors$ is of infinite type and $\torsfree$ is of finite type.
\end{enumerate}

\subsubsection{Case 1: $\tors$ finite type}\label{subsec: tors finite type}
When $\tors$ is of finite type, we know by \Cref{cor:equivalence when ext-finite} that we have 
\[
\frac{\extort\tors}{\extort\tors\cap\tors}\simeq \torsfree
\]
since the tubes are Ext-finite categories. Let us illustrate how this looks explicitly in the tubes.

When $\tors$ is of finite type, we know by \cite[Prop.~5.12]{baur2014torsion} that torsion pairs are given by a collection of rays and torsion pairs in wings below the rays. More explicitly,
we can find indices $0\leq i_0<i_1<\ldots<i_{k-1}$ with the convention $i_k=i_0+r$ giving that 
\[\torsfree=\add{\left(\bigcup_{\alpha=0}^{k-1}\left(\ray{i_\alpha}\cup \torsfree_\alpha\right)\right)}
\quad\text{ and }\quad
\tors=\add\left(\bigcup_{\alpha=0}^{k-1} \tors_\alpha\right)\]
where $(\tors_\alpha,\,\torsfree_\alpha)$ are torsion pairs in $\wing{i_\alpha}{i_{\alpha+1}-1}$ such that the Ext-projectives of $\wing{i_{\alpha}}{i_{\alpha+1}-1}$ lie in $\torsfree_\alpha$ and the remaining Ext-injectives lie in $\tors_\alpha$. 

\begin{example}\label{ex:torsion-finiteWings}
Consider a tube $\tube$ of rank $5$, with a torsion pair $\torspair$ such that the rays $\ray{0}$ and $\ray{4}$ lie in $\torsfree$. In the wing $\wing{0}{3}$ we then have $\ind{1}{3},\ind{2}{3}$ and $\ind{3}{3}$ in $\tors$. $\ind{1}{1},\ind{1}{2}$ and $\ind{2}{2}$ can lie in either $\tors$ or $\torsfree$. The remaining indecomposables lie in neither. In the illustration below we have drawn in the wing $\wing{0}{3}$, marked the torsion-free objects with a filled circle and the torsion objects with a square. The crossed objects lie in neither and the non-filled circles may lie in either.
    
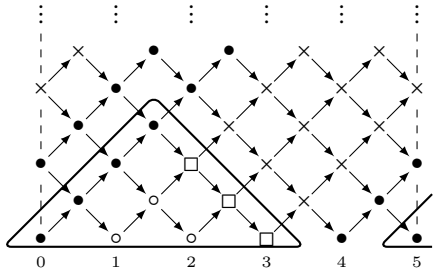
\[
\begin{tikzpicture}
    \node (S0) at (0,0){$\bullet$};
    \node (S1) at (1,0) {$\circ$};
    \node (S2) at (2,0) {$\circ$};
    \node (S3) at (3,0) {$\square$};
    \node (S4) at (4,0) {$\bullet$};
    \node (S5) at (5,0) {$\bullet$};
    \node (r2-1) at (.5,.5) {$\bullet$};
    \node (r2-2) at (1.5,.5) {$\circ$};
    \node (r2-3) at (2.5,.5) {$\square$};
    \node (r2-4) at (3.5,.5) {$\times$};
    \node (r2-5) at (4.5,.5) {$\bullet$};
    \node (r3-0) at (0,1) {$\bullet$};
    \node (r3-1) at (1,1) {$\bullet$};
    \node (r3-2) at (2,1) {$\square$};
    \node (r3-3) at (3,1) {$\times$};
    \node(r3-4) at (4,1) {$\times$};
    \node (r3-5) at (5,1){$\bullet$};
    \node (r4-1) at (.5,1.5) {$\bullet$};
    \node (r4-2) at (1.5,1.5) {$\bullet$};
    \node (r4-3) at (2.5,1.5) {$\times$};
    \node (r4-4) at (3.5,1.5) {$\times$};
    \node (r4-5) at (4.5,1.5) {$\times$};
    \node (r5-0) at (0,2) {$\times$};
    \node (r5-1) at (1,2) {$\bullet$};
    \node (r5-2) at (2,2) {$\bullet$};
    \node (r5-3) at (3,2) {$\times$};
    \node (r5-4) at (4,2) {$\times$};
    \node (r5-5) at (5,2) {$\times$};
    \node (r6-1) at (.5,2.5) {$\times$};
    \node (r6-2) at (1.5,2.5) {$\bullet$};
    \node (r6-3) at (2.5,2.5) {$\bullet$};
    \node (r6-4) at (3.5,2.5) {$\times$};
    \node (r6-5) at (4.5,2.5) {$\times$};

    \foreach \x in {0,1,...,5}{
    
    \node (r7-\x) at (\x,3)[]{$\vdots$};
    }
    \foreach \x in {0,1,...,4}{
    \node at (S\x) [below,yshift=-3pt]{\tiny$\x$};
    }
    \node at (S5) [below,yshift=-3pt]{\tiny$5$};
    \draw[dashed] (S0)--(r3-0);
    \draw[dashed] (r3-0)--(r5-0);
    \draw[dashed] (r5-0)--(r7-0);
    \draw[dashed] (S5)--(r3-5);
    \draw[dashed] (r3-5)--(r5-5);
    \draw[dashed] (r5-5)--(r7-5);
    \foreach \x [evaluate=\x as \xx using int(\x+1)] in {0,1,2,3,4}{
        \draw[-latex,opacity=.3] ($(S\x)+(.1,.1)$)--($(r2-\xx)-(.1,.1)$);
        \draw[-latex,opacity=.3] ($(r3-\x)+(.1,.1)$)--($(r4-\xx)-(.1,.1)$);
        \draw[-latex,opacity=.3] ($(r3-\x)+(.1,-.1)$)--($(r2-\xx)+(-.1,.1)$);
        \draw[-latex,opacity=.3] ($(r5-\x)+(.1,-.1)$)--($(r4-\xx)+(-.1,.1)$);
        \draw[-latex,opacity=.3] ($(r5-\x)+(.1,.1)$)--($(r6-\xx)-(.1,.1)$);

        \draw[-latex,opacity=.3] ($(r2-\xx)+(.1,-.1)$)--($(S\xx)+(-.1,.1)$);
        \draw[-latex,opacity=.3] ($(r2-\xx)+(.1,.1)$)--($(r3-\xx)-(.1,.1)$);
        \draw[-latex,opacity=.3] ($(r4-\xx)+(.1,-.1)$)--($(r3-\xx)+(-.1,.1)$);
        \draw[-latex,opacity=.3] ($(r4-\xx)+(.1,.1)$)--($(r5-\xx)-(.1,.1)$);
        \draw[-latex,opacity=.3] ($(r6-\xx)+(.1,-.1)$)--($(r5-\xx)+(-.1,.1)$);
    }
    \draw[thick,rounded corners=3pt] ($(r4-2.center)+(0,.4)$)--($(S0.center)-(.5,.1)$)--($(S3.center)+(.5,-.1)$)--cycle;
    \draw[thick,rounded corners=3pt] ($(S5.center)+(.2,.6)$)--($(S5.center)-(.5,.1)$)--+(.7,0);
    
    \end{tikzpicture}\qedhere
    \]
\end{example}

The wings $\wing{i_\alpha}{i_{\alpha+1}-1}$ can be seen as embeddings of the linearly oriented path algebra of type $\qA_{i_{\alpha+1}-i_\alpha}$ and thus within them the equivalence is reduced to the representation-finite case. Let us therefore rather focus on what happens outside of them. For an object $\ind{i_{\beta}}{i_{\beta}+l-1}\in\ray{i_{\beta}}$ with $\beta=\alpha+1$ we have
\[
0\to \ind{i_{\alpha}+1}{i_{\beta}-1}\to\ind{i_{\alpha}+1}{i_{\beta}+l-1}\to \ind{i_{\beta}}{i_{\beta}+l-1}\to 0
\]
as universal extension to $\tors$. Hence, under the functor $\cotorsfunc\colon \univTorsfree\to \extort\tors/\mathcal{I}$ of \Cref{Lemma:functorial wakamatsu} the ray $\ray{i_\beta}$ is sent to 
\[ \cotorsfunc\ray{i_{\beta}}=\add\{\ind{i_{\alpha}+1}{i_{\beta}}\to\ind{i_{\alpha}+1}{i_{\beta}+1}\to\ind{i_{\alpha}+1}{i_{\beta}+2}\to\cdots\}\subseteq \tau^{-}\ray{i_\alpha+1}. \]
with equality when $i_{\beta}=i_\alpha+1$.
\begin{example}[\Cref{ex:torsion-finiteWings} continued]
    Consider the torsion pair of \Cref{ex:torsion-finiteWings} where we also have chosen $\ind{1}{1}$, $\ind{1}{2}$ and $\ind{2}{2}$ to be torsion-free. Then we see that the indecomposables of $\extort\tors\setminus \tors$ are given by the nodes in the shaded region of the left hand side illustration. In the right hand side illustration the shaded nodes are those representing indecomposables in $\univTorsfree$.
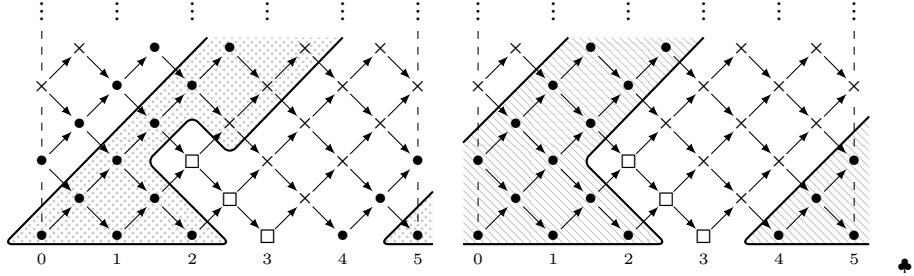
    \[\begin{tikzpicture}
    \node (S0) at (0,0){$\bullet$};
    \node (S1) at (1,0) {$\bullet$};
    \node (S2) at (2,0) {$\bullet$};
    \node (S3) at (3,0) {$\square$};
    \node (S4) at (4,0) {$\bullet$};
    \node (S5) at (5,0) {$\bullet$};
    \node (r2-1) at (.5,.5) {$\bullet$};
    \node (r2-2) at (1.5,.5) {$\bullet$};
    \node (r2-3) at (2.5,.5) {$\square$};
    \node (r2-4) at (3.5,.5) {$\times$};
    \node (r2-5) at (4.5,.5) {$\bullet$};
    \node (r3-0) at (0,1) {$\bullet$};
    \node (r3-1) at (1,1) {$\bullet$};
    \node (r3-2) at (2,1) {$\square$};
    \node (r3-3) at (3,1) {$\times$};
    \node(r3-4) at (4,1) {$\times$};
    \node (r3-5) at (5,1){$\bullet$};
    \node (r4-1) at (.5,1.5) {$\bullet$};
    \node (r4-2) at (1.5,1.5) {$\bullet$};
    \node (r4-3) at (2.5,1.5) {$\times$};
    \node (r4-4) at (3.5,1.5) {$\times$};
    \node (r4-5) at (4.5,1.5) {$\times$};
    \node (r5-0) at (0,2) {$\times$};
    \node (r5-1) at (1,2) {$\bullet$};
    \node (r5-2) at (2,2) {$\bullet$};
    \node (r5-3) at (3,2) {$\times$};
    \node (r5-4) at (4,2) {$\times$};
    \node (r5-5) at (5,2) {$\times$};
    \node (r6-1) at (.5,2.5) {$\times$};
    \node (r6-2) at (1.5,2.5) {$\bullet$};
    \node (r6-3) at (2.5,2.5) {$\bullet$};
    \node (r6-4) at (3.5,2.5) {$\times$};
    \node (r6-5) at (4.5,2.5) {$\times$};

    \foreach \x in {0,1,...,5}{
    
    \node (r7-\x) at (\x,3)[]{$\vdots$};
    }
    \foreach \x in {0,1,...,4}{
    \node at (S\x) [below,yshift=-3pt]{\tiny$\x$};
    }
    \node at (S5) [below,yshift=-3pt]{\tiny$5$};
    \draw[dashed] (S0)--(r3-0);
    \draw[dashed] (r3-0)--(r5-0);
    \draw[dashed] (r5-0)--(r7-0);
    \draw[dashed] (S5)--(r3-5);
    \draw[dashed] (r3-5)--(r5-5);
    \draw[dashed] (r5-5)--(r7-5);
    \foreach \x [evaluate=\x as \xx using int(\x+1)] in {0,1,2,3,4}{
        \draw[-latex,opacity=.3] ($(S\x)+(.1,.1)$)--($(r2-\xx)-(.1,.1)$);
        \draw[-latex,opacity=.3] ($(r3-\x)+(.1,.1)$)--($(r4-\xx)-(.1,.1)$);
        \draw[-latex,opacity=.3] ($(r3-\x)+(.1,-.1)$)--($(r2-\xx)+(-.1,.1)$);
        \draw[-latex,opacity=.3] ($(r5-\x)+(.1,-.1)$)--($(r4-\xx)+(-.1,.1)$);
        \draw[-latex,opacity=.3] ($(r5-\x)+(.1,.1)$)--($(r6-\xx)-(.1,.1)$);

        \draw[-latex,opacity=.3] ($(r2-\xx)+(.1,-.1)$)--($(S\xx)+(-.1,.1)$);
        \draw[-latex,opacity=.3] ($(r2-\xx)+(.1,.1)$)--($(r3-\xx)-(.1,.1)$);
        \draw[-latex,opacity=.3] ($(r4-\xx)+(.1,-.1)$)--($(r3-\xx)+(-.1,.1)$);
        \draw[-latex,opacity=.3] ($(r4-\xx)+(.1,.1)$)--($(r5-\xx)-(.1,.1)$);
        \draw[-latex,opacity=.3] ($(r6-\xx)+(.1,-.1)$)--($(r5-\xx)+(-.1,.1)$);
    }
    \draw[thick,rounded corners=3pt] ($(r6-3.center)+(-.3,.15)$)--($(S0.center)-(.5,.1)$)--($(S2.center)+(.5,-.1)$)--($(r3-1.center)+(.4,0)$)--($(r5-2.center)-(0,.4)$)--($(r4-3.center)-(0,.4)$)--($(r6-4.center)+(.5,.15)$);
    \draw[thick,rounded corners=3pt] ($(S5.center)+(.2,.6)$)--($(S5.center)-(.5,.1)$)--+(.7,0);
    
    \begin{scope}[on background layer]
        \fill[pattern=crosshatch dots,pattern color=gray!40,rounded corners=3pt] ($(r6-3.center)+(-.3,.15)$)--($(S0.center)-(.5,.1)$)--($(S2.center)+(.5,-.1)$)--($(r3-1.center)+(.4,0)$)--($(r5-2.center)-(0,.4)$)--($(r4-3.center)-(0,.4)$)--($(r6-4.center)+(.5,.15)$)--cycle;
        \fill[pattern=crosshatch dots,pattern color=gray!40,rounded corners=3pt] ($(S5.center)+(.2,.6)$)--($(S5.center)-(.5,.1)$)--+(.7,0)--cycle;
        
    \end{scope}

    \end{tikzpicture}
    \quad
    \begin{tikzpicture}
    \node (S0) at (0,0){$\bullet$};
    \node (S1) at (1,0) {$\bullet$};
    \node (S2) at (2,0) {$\bullet$};
    \node (S3) at (3,0) {$\square$};
    \node (S4) at (4,0) {$\bullet$};
    \node (S5) at (5,0) {$\bullet$};
    \node (r2-1) at (.5,.5) {$\bullet$};
    \node (r2-2) at (1.5,.5) {$\bullet$};
    \node (r2-3) at (2.5,.5) {$\square$};
    \node (r2-4) at (3.5,.5) {$\times$};
    \node (r2-5) at (4.5,.5) {$\bullet$};
    \node (r3-0) at (0,1) {$\bullet$};
    \node (r3-1) at (1,1) {$\bullet$};
    \node (r3-2) at (2,1) {$\square$};
    \node (r3-3) at (3,1) {$\times$};
    \node(r3-4) at (4,1) {$\times$};
    \node (r3-5) at (5,1){$\bullet$};
    \node (r4-1) at (.5,1.5) {$\bullet$};
    \node (r4-2) at (1.5,1.5) {$\bullet$};
    \node (r4-3) at (2.5,1.5) {$\times$};
    \node (r4-4) at (3.5,1.5) {$\times$};
    \node (r4-5) at (4.5,1.5) {$\times$};
    \node (r5-0) at (0,2) {$\times$};
    \node (r5-1) at (1,2) {$\bullet$};
    \node (r5-2) at (2,2) {$\bullet$};
    \node (r5-3) at (3,2) {$\times$};
    \node (r5-4) at (4,2) {$\times$};
    \node (r5-5) at (5,2) {$\times$};
    \node (r6-1) at (.5,2.5) {$\times$};
    \node (r6-2) at (1.5,2.5) {$\bullet$};
    \node (r6-3) at (2.5,2.5) {$\bullet$};
    \node (r6-4) at (3.5,2.5) {$\times$};
    \node (r6-5) at (4.5,2.5) {$\times$};

    \foreach \x in {0,1,...,5}{
    
    \node (r7-\x) at (\x,3)[]{$\vdots$};
    }
    \foreach \x in {0,1,...,4}{
    \node at (S\x) [below,yshift=-3pt]{\tiny$\x$};
    }
    \node at (S5) [below,yshift=-3pt]{\tiny$5$};
    \draw[dashed] (S0)--(r3-0);
    \draw[dashed] (r3-0)--(r5-0);
    \draw[dashed] (r5-0)--(r7-0);
    \draw[dashed] (S5)--(r3-5);
    \draw[dashed] (r3-5)--(r5-5);
    \draw[dashed] (r5-5)--(r7-5);
    \foreach \x [evaluate=\x as \xx using int(\x+1)] in {0,1,2,3,4}{
        \draw[-latex,opacity=.3] ($(S\x)+(.1,.1)$)--($(r2-\xx)-(.1,.1)$);
        \draw[-latex,opacity=.3] ($(r3-\x)+(.1,.1)$)--($(r4-\xx)-(.1,.1)$);
        \draw[-latex,opacity=.3] ($(r3-\x)+(.1,-.1)$)--($(r2-\xx)+(-.1,.1)$);
        \draw[-latex,opacity=.3] ($(r5-\x)+(.1,-.1)$)--($(r4-\xx)+(-.1,.1)$);
        \draw[-latex,opacity=.3] ($(r5-\x)+(.1,.1)$)--($(r6-\xx)-(.1,.1)$);

        \draw[-latex,opacity=.3] ($(r2-\xx)+(.1,-.1)$)--($(S\xx)+(-.1,.1)$);
        \draw[-latex,opacity=.3] ($(r2-\xx)+(.1,.1)$)--($(r3-\xx)-(.1,.1)$);
        \draw[-latex,opacity=.3] ($(r4-\xx)+(.1,-.1)$)--($(r3-\xx)+(-.1,.1)$);
        \draw[-latex,opacity=.3] ($(r4-\xx)+(.1,.1)$)--($(r5-\xx)-(.1,.1)$);
        \draw[-latex,opacity=.3] ($(r6-\xx)+(.1,-.1)$)--($(r5-\xx)+(-.1,.1)$);
    }
    \draw[thick,rounded corners=3pt] {($(r6-2.center)+(-.3,.15)$)--($(r3-0.center)-(.2,-.25)$)}{($(S0.center)-(.2,.1)$)--($(S2.center)+(.5,-.1)$)--($(r3-1.center)+(.4,0)$)--($(r6-3.center)+(.5,.15)$)};
    \draw[thick,rounded corners=3pt] ($(r3-5.center)+(.2,.6)$)--($(S4.center)-(.5,.1)$)--($(S5.center)-(0,.1)$)--+(.2,0);
    \begin{scope}[on background layer]
        \fill[pattern=north west lines,pattern color=gray!40,rounded corners=3pt] 
        ($(r6-2.center)+(-.3,.15)$)--($(r3-0.center)-(.2,-.25)$)--
        ($(S0.center)-(.2,.1)$)--($(S2.center)+(.5,-.1)$)--($(r3-1.center)+(.4,0)$)--($(r6-3.center)+(.5,.15)$)
        --cycle;
        \fill[pattern=north west lines,pattern color=gray!40,rounded corners=3pt] ($(r3-5.center)+(.2,.6)$)--($(S4.center)-(.5,.1)$)--($(S5.center)-(0,.1)$)--+(.2,0);
    \end{scope}
    
    \end{tikzpicture}\qedhere
    \]
    
\end{example}

\subsubsection{Case 2: $\tors$ infinite type}\label{subsec: tors infinite type}

When $\tors$ is infinite, the situation changes quite drastically. In this setting our torsion pair is built from a collection of corays and torsion pairs in the wings below these. More precisely, by \cite[Prop.~5.12]{baur2014torsion} we can find indices $0\leq i_0<i_1<\cdots<i_{k-1}$ with the convention $i_k=i_0+r$ such that
\[
\tors=\add\left(\bigcup_{\alpha=0}^{k-1}(\coray{i_\alpha}\cup \tors_\alpha)\right)
\quad\text{ and }\quad
\torsfree=\add\left(\bigcup_{\alpha=0}^{k-1}\torsfree_\alpha\right)
\]
where $(\tors_\alpha,\,\torsfree_\alpha)$ are torsion pairs in $\wing{i_\alpha+1}{i_{\alpha+1}}$ with the Ext-injectives of $\wing{i_\alpha+1}{i_{\alpha+1}}$ in $\tors_\alpha$ and the remaining Ext-projectives in $\torsfree_\alpha$.

\begin{example}
Consider again the tube $\tube$ of rank $5$, but now with a torsion pair $\torspair$ such that the coray $\coray{0}$ lie in $\tors$. Then in $\wing{1}{5}$ we have $\ind{1}{1},\ind{1}{2}$, $\ind{1}{3}$ and $\ind{1}{4}$ in $\torsfree$. The indecomposables $\ind{2}{2},\ind{2}{3}$, $\ind{2}{4}$, $\ind{3}{3}$, $\ind{3}{4}$ and $\ind{4}{4}$ can lie in $\tors$ or $\torsfree$. The remaining indecomposables lie in neither. We follow the same convention as above in the illustration; torsion objects are given by squares, torsion-free objects by filled circles, non-declared objects by non-filled circles and the excluded objects by crosses.
    
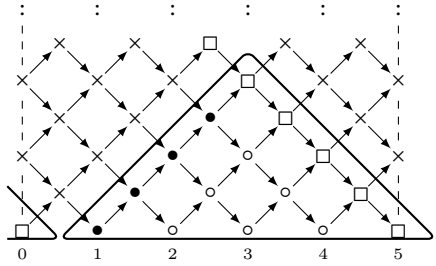
\[
\begin{tikzpicture}
    \node (S0) at (0,0){$\square$};
    \node (S1) at (1,0) {$\bullet$};
    \node (S2) at (2,0) {$\circ$};
    \node (S3) at (3,0) {$\circ$};
    \node (S4) at (4,0) {$\circ$};
    \node (S5) at (5,0) {$\square$};
    
    \node (r2-1) at (.5,.5) {$\times$};
    \node (r2-2) at (1.5,.5) {$\bullet$};
    \node (r2-3) at (2.5,.5) {$\circ$};
    \node (r2-4) at (3.5,.5) {$\circ$};
    \node (r2-5) at (4.5,.5) {$\square$};
    
    \node (r3-0) at (0,1) {$\times$};
    \node (r3-1) at (1,1) {$\times$};
    \node (r3-2) at (2,1) {$\bullet$};
    \node (r3-3) at (3,1) {$\circ$};
    \node(r3-4) at (4,1) {$\square$};
    \node (r3-5) at (5,1){$\times$};
    
    \node (r4-1) at (.5,1.5) {$\times$};
    \node (r4-2) at (1.5,1.5) {$\times$};
    \node (r4-3) at (2.5,1.5) {$\bullet$};
    \node (r4-4) at (3.5,1.5) {$\square$};
    \node (r4-5) at (4.5,1.5) {$\times$};
    
    \node (r5-0) at (0,2) {$\times$};
    \node (r5-1) at (1,2) {$\times$};
    \node (r5-2) at (2,2) {$\times$};
    \node (r5-3) at (3,2) {$\square$};
    \node (r5-4) at (4,2) {$\times$};
    \node (r5-5) at (5,2) {$\times$};
    
    \node (r6-1) at (.5,2.5) {$\times$};
    \node (r6-2) at (1.5,2.5) {$\times$};
    \node (r6-3) at (2.5,2.5) {$\square$};
    \node (r6-4) at (3.5,2.5) {$\times$};
    \node (r6-5) at (4.5,2.5) {$\times$};

    \foreach \x in {0,1,...,5}{
    
    \node (r7-\x) at (\x,3)[]{$\vdots$};
    }
    \foreach \x in {0,1,...,4}{
    \node at (S\x) [below,yshift=-3pt]{\tiny$\x$};
    }
    \node at (S5) [below,yshift=-3pt]{\tiny$5$};
    \draw[dashed] (S0)--(r3-0);
    \draw[dashed] (r3-0)--(r5-0);
    \draw[dashed] (r5-0)--(r7-0);
    \draw[dashed] (S5)--(r3-5);
    \draw[dashed] (r3-5)--(r5-5);
    \draw[dashed] (r5-5)--(r7-5);
    \foreach \x [evaluate=\x as \xx using int(\x+1)] in {0,1,2,3,4}{
        \draw[-latex,opacity=.3] ($(S\x)+(.1,.1)$)--($(r2-\xx)-(.1,.1)$);
        \draw[-latex,opacity=.3] ($(r3-\x)+(.1,.1)$)--($(r4-\xx)-(.1,.1)$);
        \draw[-latex,opacity=.3] ($(r3-\x)+(.1,-.1)$)--($(r2-\xx)+(-.1,.1)$);
        \draw[-latex,opacity=.3] ($(r5-\x)+(.1,-.1)$)--($(r4-\xx)+(-.1,.1)$);
        \draw[-latex,opacity=.3] ($(r5-\x)+(.1,.1)$)--($(r6-\xx)-(.1,.1)$);

        \draw[-latex,opacity=.3] ($(r2-\xx)+(.1,-.1)$)--($(S\xx)+(-.1,.1)$);
        \draw[-latex,opacity=.3] ($(r2-\xx)+(.1,.1)$)--($(r3-\xx)-(.1,.1)$);
        \draw[-latex,opacity=.3] ($(r4-\xx)+(.1,-.1)$)--($(r3-\xx)+(-.1,.1)$);
        \draw[-latex,opacity=.3] ($(r4-\xx)+(.1,.1)$)--($(r5-\xx)-(.1,.1)$);
        \draw[-latex,opacity=.3] ($(r6-\xx)+(.1,-.1)$)--($(r5-\xx)+(-.1,.1)$);
    }
    \draw[thick,rounded corners=3pt] ($(r5-3.center)+(0,.4)$)--($(S1.center)-(.5,.1)$)--($(S5.center)+(.5,-.1)$)--cycle;

    \draw[thick,rounded corners=3pt] ($(S0.center)+(-.2,.6)$)--($(S0.center)+(.5,-.1)$)-- +(-.7,0);
    \end{tikzpicture}
    \]
    Within the wing $\wing{2}{4}$, we observe that all the indecomposables have only the trivial extension to the coray $\coray{0}$. Hence, we can without obstruction study the equivalence of \Cref{thm: Main Theorem} as if we were in the module category of the path algebra over the linearly oriented dynkin quiver $\qA_3$. Which means we have $\torsfree\cap \wing{2}{4}=\univTorsfree\cap\wing{2}{4}\simeq (\extort\tors\cap\wing{2}{4})/\mathcal{I}$. 

    Observe now that the Ext-projectives of $\wing{1}{5}$ all have extensions to all indecomposables of $\coray{0}$. In particular they are not ext-orthogonal to $\tors$. Further, since the indecomposables in $\coray{0}$ have arbitrarily large lengths we realise through the push-out description of universal extensions in \Cref{remark:universal extension as pushout} that the Ext-projectives also do not admit any universal extensions to $\tors$. In summary, we have $\univTorsfree=\torsfree\cap\wing{2}{4}\subsetneq \torsfree$.  
    
    From the discussion above, we observe that $\extort\tors\subseteq \wing{2}{5}$. Now, using that $\wing{2}{5}$ is an exact embedding of $\K\qA_4$ and every module over that algebra has a universal extension to every subcategory, we conclude as in \Cref{Lem:func fin tors give old equivalence} that the quotient $\extort\tors/\mathcal{I}$ is equivalent to $\extort\tors/[\tors\cap\extort\tors]$.

    In the tube $\tube$ we conclude that the equivalence of \Cref{thm: Main Theorem} is given by
    \[
    \torsfree\cap\wing{2}{4}\simeq \extort\tors/[\tors\cap\extort\tors].
    \]
    We give an illustration below where we have made a choice in $\wing{2}{4}$. In the left picture we have shaded the indecomposable in $\extort\tors\setminus\tors$ and in the right picture we have shaded the indecomposable in $\univTorsfree$.

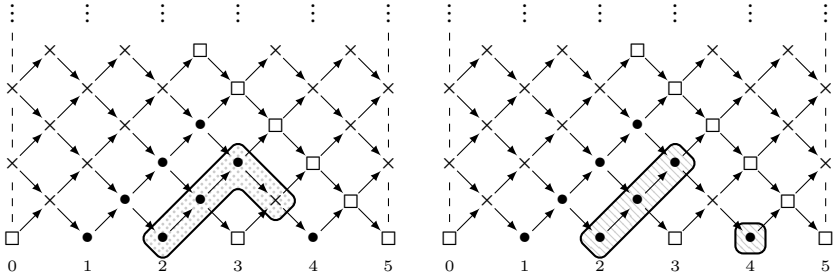
    \[
    \begin{tikzpicture}
    \node (S0) at (0,0){$\square$};
    \node (S1) at (1,0) {$\bullet$};
    \node (S2) at (2,0) {$\bullet$};
    \node (S3) at (3,0) {$\square$};
    \node (S4) at (4,0) {$\bullet$};
    \node (S5) at (5,0) {$\square$};
    
    \node (r2-1) at (.5,.5) {$\times$};
    \node (r2-2) at (1.5,.5) {$\bullet$};
    \node (r2-3) at (2.5,.5) {$\bullet$};
    \node (r2-4) at (3.5,.5) {$\times$};
    \node (r2-5) at (4.5,.5) {$\square$};
    
    \node (r3-0) at (0,1) {$\times$};
    \node (r3-1) at (1,1) {$\times$};
    \node (r3-2) at (2,1) {$\bullet$};
    \node (r3-3) at (3,1) {$\bullet$};
    \node(r3-4) at (4,1) {$\square$};
    \node (r3-5) at (5,1){$\times$};
    
    \node (r4-1) at (.5,1.5) {$\times$};
    \node (r4-2) at (1.5,1.5) {$\times$};
    \node (r4-3) at (2.5,1.5) {$\bullet$};
    \node (r4-4) at (3.5,1.5) {$\square$};
    \node (r4-5) at (4.5,1.5) {$\times$};
    
    \node (r5-0) at (0,2) {$\times$};
    \node (r5-1) at (1,2) {$\times$};
    \node (r5-2) at (2,2) {$\times$};
    \node (r5-3) at (3,2) {$\square$};
    \node (r5-4) at (4,2) {$\times$};
    \node (r5-5) at (5,2) {$\times$};
    
    \node (r6-1) at (.5,2.5) {$\times$};
    \node (r6-2) at (1.5,2.5) {$\times$};
    \node (r6-3) at (2.5,2.5) {$\square$};
    \node (r6-4) at (3.5,2.5) {$\times$};
    \node (r6-5) at (4.5,2.5) {$\times$};

    \foreach \x in {0,1,...,5}{
    
    \node (r7-\x) at (\x,3)[]{$\vdots$};
    }
    \foreach \x in {0,1,...,4}{
    \node at (S\x) [below,yshift=-5pt]{\tiny$\x$};
    }
    \node at (S5) [below,yshift=-5pt]{\tiny$5$};
    \draw[dashed] (S0)--(r3-0);
    \draw[dashed] (r3-0)--(r5-0);
    \draw[dashed] (r5-0)--(r7-0);
    \draw[dashed] (S5)--(r3-5);
    \draw[dashed] (r3-5)--(r5-5);
    \draw[dashed] (r5-5)--(r7-5);
    \foreach \x [evaluate=\x as \xx using int(\x+1)] in {0,1,2,3,4}{
        \draw[-latex,opacity=.3] ($(S\x)+(.1,.1)$)--($(r2-\xx)-(.1,.1)$);
        \draw[-latex,opacity=.3] ($(r3-\x)+(.1,.1)$)--($(r4-\xx)-(.1,.1)$);
        \draw[-latex,opacity=.3] ($(r3-\x)+(.1,-.1)$)--($(r2-\xx)+(-.1,.1)$);
        \draw[-latex,opacity=.3] ($(r5-\x)+(.1,-.1)$)--($(r4-\xx)+(-.1,.1)$);
        \draw[-latex,opacity=.3] ($(r5-\x)+(.1,.1)$)--($(r6-\xx)-(.1,.1)$);

        \draw[-latex,opacity=.3] ($(r2-\xx)+(.1,-.1)$)--($(S\xx)+(-.1,.1)$);
        \draw[-latex,opacity=.3] ($(r2-\xx)+(.1,.1)$)--($(r3-\xx)-(.1,.1)$);
        \draw[-latex,opacity=.3] ($(r4-\xx)+(.1,-.1)$)--($(r3-\xx)+(-.1,.1)$);
        \draw[-latex,opacity=.3] ($(r4-\xx)+(.1,.1)$)--($(r5-\xx)-(.1,.1)$);
        \draw[-latex,opacity=.3] ($(r6-\xx)+(.1,-.1)$)--($(r5-\xx)+(-.1,.1)$);
    }
    \begin{scope}[on background layer]
        \fill[pattern=crosshatch dots,pattern color=gray!40,rounded corners=3pt] ($(r3-3.center)+(0,.3)$)--($(S2.center)-(.3,0)$)--($(S2.center)+(0,-.3)$)--($(r3-3.center)-(0,0.3)$)--($(r2-4.center)-(0,.3)$)--($(r2-4.center)+(0.3,0)$)--cycle;
        
    \end{scope}
    \draw[thick,rounded corners=3pt] ($(r3-3.center)+(0,.3)$)--($(S2.center)-(.3,0)$)--($(S2.center)+(0,-.3)$)--($(r3-3.center)-(0,0.3)$)--($(r2-4.center)-(0,.3)$)--($(r2-4.center)+(0.3,0)$)--cycle;

    \end{tikzpicture}
    \quad
    \begin{tikzpicture}
    \node (S0) at (0,0){$\square$};
    \node (S1) at (1,0) {$\bullet$};
    \node (S2) at (2,0) {$\bullet$};
    \node (S3) at (3,0) {$\square$};
    \node (S4) at (4,0) {$\bullet$};
    \node (S5) at (5,0) {$\square$};
    
    \node (r2-1) at (.5,.5) {$\times$};
    \node (r2-2) at (1.5,.5) {$\bullet$};
    \node (r2-3) at (2.5,.5) {$\bullet$};
    \node (r2-4) at (3.5,.5) {$\times$};
    \node (r2-5) at (4.5,.5) {$\square$};
    
    \node (r3-0) at (0,1) {$\times$};
    \node (r3-1) at (1,1) {$\times$};
    \node (r3-2) at (2,1) {$\bullet$};
    \node (r3-3) at (3,1) {$\bullet$};
    \node(r3-4) at (4,1) {$\square$};
    \node (r3-5) at (5,1){$\times$};
    
    \node (r4-1) at (.5,1.5) {$\times$};
    \node (r4-2) at (1.5,1.5) {$\times$};
    \node (r4-3) at (2.5,1.5) {$\bullet$};
    \node (r4-4) at (3.5,1.5) {$\square$};
    \node (r4-5) at (4.5,1.5) {$\times$};
    
    \node (r5-0) at (0,2) {$\times$};
    \node (r5-1) at (1,2) {$\times$};
    \node (r5-2) at (2,2) {$\times$};
    \node (r5-3) at (3,2) {$\square$};
    \node (r5-4) at (4,2) {$\times$};
    \node (r5-5) at (5,2) {$\times$};
    
    \node (r6-1) at (.5,2.5) {$\times$};
    \node (r6-2) at (1.5,2.5) {$\times$};
    \node (r6-3) at (2.5,2.5) {$\square$};
    \node (r6-4) at (3.5,2.5) {$\times$};
    \node (r6-5) at (4.5,2.5) {$\times$};

    \foreach \x in {0,1,...,5}{
    
    \node (r7-\x) at (\x,3)[]{$\vdots$};
    }
    \foreach \x in {0,1,...,4}{
    \node at (S\x) [below,yshift=-5pt]{\tiny$\x$};
    }
    \node at (S5) [below,yshift=-5pt]{\tiny$5$};
    \draw[dashed] (S0)--(r3-0);
    \draw[dashed] (r3-0)--(r5-0);
    \draw[dashed] (r5-0)--(r7-0);
    \draw[dashed] (S5)--(r3-5);
    \draw[dashed] (r3-5)--(r5-5);
    \draw[dashed] (r5-5)--(r7-5);
    \foreach \x [evaluate=\x as \xx using int(\x+1)] in {0,1,2,3,4}{
        \draw[-latex,opacity=.3] ($(S\x)+(.1,.1)$)--($(r2-\xx)-(.1,.1)$);
        \draw[-latex,opacity=.3] ($(r3-\x)+(.1,.1)$)--($(r4-\xx)-(.1,.1)$);
        \draw[-latex,opacity=.3] ($(r3-\x)+(.1,-.1)$)--($(r2-\xx)+(-.1,.1)$);
        \draw[-latex,opacity=.3] ($(r5-\x)+(.1,-.1)$)--($(r4-\xx)+(-.1,.1)$);
        \draw[-latex,opacity=.3] ($(r5-\x)+(.1,.1)$)--($(r6-\xx)-(.1,.1)$);

        \draw[-latex,opacity=.3] ($(r2-\xx)+(.1,-.1)$)--($(S\xx)+(-.1,.1)$);
        \draw[-latex,opacity=.3] ($(r2-\xx)+(.1,.1)$)--($(r3-\xx)-(.1,.1)$);
        \draw[-latex,opacity=.3] ($(r4-\xx)+(.1,-.1)$)--($(r3-\xx)+(-.1,.1)$);
        \draw[-latex,opacity=.3] ($(r4-\xx)+(.1,.1)$)--($(r5-\xx)-(.1,.1)$);
        \draw[-latex,opacity=.3] ($(r6-\xx)+(.1,-.1)$)--($(r5-\xx)+(-.1,.1)$);
    }
    \begin{scope}[on background layer]
        \fill[pattern=north west lines,pattern color=gray!40,rounded corners=3pt] ($(S4.center)-(0.2,.2)$)--($(S4.center)+(0.2,-.2)$)--($(S4.center)+(0.2,.2)$)--
        ($(S4.center)+(-0.2,.2)$)--cycle;
        \fill[pattern=north west lines,pattern color=gray!40,rounded corners=3pt] ($(r3-3.center)+(0,.3)$)--($(S2.center)-(.3,0)$)--($(S2.center)+(0,-.3)$)--($(r3-3.center)+(0.3,0)$)--cycle;
        
    \end{scope}
    
    \draw[thick,rounded corners=3pt] ($(r3-3.center)+(0,.3)$)--($(S2.center)-(.3,0)$)--($(S2.center)+(0,-.3)$)--($(r3-3.center)+(0.3,0)$)--cycle;
    \draw[thick,rounded corners=3pt] ($(S4.center)-(0.2,.2)$)--($(S4.center)+(0.2,-.2)$)--($(S4.center)+(0.2,.2)$)--
        ($(S4.center)+(-0.2,.2)$)--cycle;
    \end{tikzpicture}
    \]
    Note that the torsion pair in $\wing{2}{4}$ is exactly that of \Cref{ex:torsion pair A3} under an embedding of $\K\qA_3$ into $\wing{2}{4}$. However, as the Ext-projectives in $\wing{1}{4}$ are torsion-free, but not in $\univTorsfree$, we observe that the equivalence of \eqref{eq:equivalence to generalize in text} does not hold.
\end{example}

The behaviour exhibited in the example above is indicative of the general picture. That is, we have 
\[
\univTorsfree=\torsfree\cap\add\left(\bigcup_{\alpha=0}^{k-1}\wing{i_{\alpha}+2}{i_{\alpha+1}-1}\right),
\]
\[
\extort\tors=\add\left(\bigcup_{\alpha=0}^{k-1}\extort\tors_\alpha\right).
\]
and
\[
\mathcal{I}=[\tors\cap\extort\tors]
\]
In other words, when describing the equivalence of \Cref{thm: Main Theorem} we can restrict ourselves to a collection of path algebras over linearly oriented Dynkin quivers of $\qA_n$ type.

\begin{remark}
    For a tame hereditary algebra we note that for any torsion pair $\torspair$ all but finitely many torsion-free indecomposables admit a universal extension to $\tors$. This follows from the above discussion and the fact that there are no non-trivial extension from a preprojective module to a non-preprojective module.
\end{remark}

\bibliographystyle{amsalpha}
\bibliography{bibliotek}
\end{document}